\documentclass[11pt]{amsart}
\usepackage[margin=1in]{geometry}
\usepackage{changepage, bm, amsmath, makecell, multirow, cite,
  setspace,amsfonts, graphicx, color}
\usepackage{mathtools}
\usepackage{amsthm}
\usepackage{verbatim}
\usepackage{algorithm}
\usepackage{algorithmic}
\numberwithin{equation}{section}
\usepackage[toc,page]{appendix}
\usepackage{caption}
\newtheorem{theorem}{Theorem}[section]
\newtheorem{remark}{Remark}[section]

\newtheorem{lemma}[theorem]{Lemma}
\newtheorem{definition}{Definition}[section]
\newtheorem{proposition}{Proposition}[section]

\begin{document}
\title{Sensitivity analysis of quasi-stationary-distributions (QSDs)}
\author{Yao Li}
\address{Yao Li: Department of Mathematics and Statistics, University
  of Massachusetts Amherst, Amherst, MA, 01002, USA}
\email{yaoli@math.umass.edu}
\author{Yaping Yuan}
\address{Yaping Yuan: Department of Mathematics and Statistics, University
  of Massachusetts Amherst, Amherst, MA, 01002, USA}
\email{yuan@math.umass.edu}

\thanks{Authors are listed in alphabetical order. Yao Li is partially supported by NSF DMS-1813246 and DMS-2108628.}


\begin{abstract}
  This paper studies the sensitivity analysis of mass-action systems against their diffusion approximations, particularly the dependence on population sizes. As a continuous time Markov chain, a mass-action system can be described by a equation driven by finite many Poisson processes, which has a diffusion approximation that can be pathwisely matched. The magnitude of noise in mass-action systems is proportional to the square root of the molecule count/population, which makes a large class of mass-action systems have quasi-stationary distributions (QSDs) instead of invariant probability measures. In this paper we modify the coupling based technique developed in \cite{dobson2021using} to estimate an upper bound of the 1-Wasserstein distance between two QSDs. Some numerical results for sensitivity with different population sizes are provided.
\end{abstract}
\maketitle
\section{Introduction}
A mass-action network is a system of finite many species and reactions whose rule of update satisfies the mass-action law. Mass-action network covers a large number of chemical reaction network, epidemiology models, and population models. At the molecule level, reactions in the mass-action network are random events that modify the state of the network according to the stoichiometric equations. The time of these random events satisfy mass-action laws. Therefore, a mass-action network can be mathematically described by a continuous-time Markov process, which is driven by finite many Poisson processes. 

The randomness in updating the network is called the demographic noise in population and epidemiology models. It is well known that demographic noise leads to finite time extinction in a very large class of population models (see for example the discussion in Section 3.1). This is because the magnitude of the demographic noise is proportional to the population size. As a result, when the population is small, in many mass-action systems, the noise could become the dominate term and leads to finite time extinction with strictly positive probability. Therefore, the asymptotic property of the mass-action network with finite time extinction is usually described by the quasi-stationary distribution (QSD), which is the conditional limiting distribution conditioning on not hitting the absorbing set yet. As discussed in \cite{li2021data}, when the extinction rate is low, the quasi-stationary distribution can be well approximated by the invariant probability measure of a modified process that artificially "pushes" the trajectory away from the extinction. 

It has been known for decades that when the population size is large, the continuous-time Markov process converges to the mass-action ordinary differential equations (ODEs). In addition, by setting up a martingale problem, one can show that the re-scaled difference between the continuous-time Markov process and the mass-action ODE converges to a stochastic differential equation. Therefore, at any finite time, the continuous-time Markov process of a mass-action network is approximated by a stochastic differential equation. This is called the diffusion approximation of a mass-action network. We refer \cite{anderson2011continuous,gibson2000efficient} for further details.

The goal of this paper is to study the sensitivity of QSDs against the diffusion approximation. We are interested in how QSDs of the Markov process and its diffusion approximation differs from each other. The motivation is that an exact simulation at the molecule level is usually computationally expensive even if the stochastic simulation algorithm (SSA) is implemented optimally  \cite{gillespie2007stochastic,li2015fast,slepoy2008constant}. It is even harder to numerically compute the QSD when the number of molecule is large. On the other hand, the simulation of a diffusion process is much easier. The technique of computing the invariant probability measure or QSD of a stochastic differential equation is also well developed \cite{li2020numerical,li2021data,zhai2022deep}. Hence it is important to have a quantitative upper bound of the difference between the QSD of a mass-action system and that of its diffusion approximation.

The way of sensitivity analysis is developed from on the coupling-based method in \cite{dobson2021using}. We need both finite time truncation error and the rate of contraction of the transition kernel of the diffusion process. The finite time error is given by the KMT algorithm in \cite{mozgunov2018review}. With the explicit construction of coupled trajectories of the Poisson process and the diffusion process, the finite time error up to fixed time $T$ can be computed. The rate of contraction is modified from the data-driven method proposed in \cite{li2020numerical}. We design a suitable coupling scheme for the modified diffusion process that regenerates from the QSD right after hitting the absorbing set. Because of the coupling inequality, the exponential tail of the coupling time can be used to estimate the rate of contraction. The sensitivity analysis is demonstrated by several numerical examples. Generally speaking, the distance between two processes is much larger for smaller volume (i.e., molecule count).

The organization of this paper is as follows. A short preliminary about reaction networks, rates for the law of mass action, Poisson process, diffusion process and coupling times is provided in section 2. Section 3 introduces the algorithms for computing the finite time error and the rate of contraction in two different cases. All numerical examples are demonstrated in section 4. Section 5 is the conclusion. All explicit expressions of Poisson process and Wiener process are shown in the appendix.

\section{Preliminary}
\subsection{Stochastic mass reaction networks}

We consider a mass action network of $K$ reactions involving $d$
distinct species, $S_1, \cdots, S_d$,
\begin{equation}
    \label{stoich}
    \sum^d_{i=1}c_{ki}S_i \rightarrow \sum^d_{i=1}c'_{ki}S_i, \ k = 1, \cdots, K
\end{equation}
where $c_{ki}$ and $c'_{ki}$ are non-negative integers that denote the
number of molecules of species $S_{i}$ consumed and produced by
reaction $k$, respectively. Let $V$ be the volume of the reaction
system. Let $X(t) = (x_{1}(t) , \cdots, x_{d}(t)) \in
\mathbb{R}^{d}$ be the state of the mass action system at time $t$,
such that the $i$-th entry of $X(t)$ represents the concentration of
species $S_{i}$, $i = 1, \cdots, d$. In other words the number of
molecules of $S_{i}$ is $V x_{i} := N_{i}$. Let $\lambda_k$ be the rate
at which the $k$th reaction occurs, that is, it gives the propensity
of the $k$-th reaction as a function of the concentrations of molecules of the chemical species.

\subsection{Rates for the law of mass action} The law of
mass action means the rate of a reaction should be proportional to the
number of distinct subsets of the participating molecules. More
precisely, the rate of reaction $k$ reads
\begin{equation*}
    \lambda_k = \kappa_k V \prod^d_{i=1}(\frac{N_i}{V})^{c_{ki}}:= Vf_k(X),
\end{equation*}
where $\kappa_k$ is a rate constant, and $N_i$ be the number of
molecule of the $i$th species in the system. Let $\Delta t \ll 1$ be a
very short time period. More precisely, given all information of the system up to time
$t$, we have
$$
  \mathbb{P}[ \mbox{ reaction } k\  \mbox{occurs in } [t, t + \Delta t)
  ] = \lambda_k \Delta t + O(\Delta t^{2}) \,.
$$

\subsection{Poisson process}
We use Poisson counting process to represent
$X(t)$, because $X(t)$ is a continuous time discrete
state Markov chain. Let $X_{i}(t)$ be $i$-th entry of $X(t)$, then 
\begin{equation*}
    X_i(t) = X_i(0)+\frac{1}{V}\sum_{k} R_{k}(t)(c'_{ki} - c_{ki}),
\end{equation*}
where $R_{k}(t)$ is the number of times the reaction $k$ has occurred
by time $t$ and $R_{k}(0)=0$. Because the number of
molecules of species changes with time, $R_{k}(t)$ is an inhomogeneous
Poisson process that is  given by 
\begin{equation}
    R_k(t) = P_k(V\int^t_0 f_k(X(s))ds),
\end{equation}
where $P_{k}(\cdot)$ is a unit-rate Poisson point process. It is well
known that $P_{k}(\cdot)$ satisfies the following three properties: (1)
$P_{k}(0) = 0$, (2) $P_{k}(\cdot)$ has independent increments, and (3)
$P_{k}(s+t) - P_{k}(s)$ is a Poisson random variable with parameter
$t$. And the whole system is given by
\begin{equation}
\label{poss}
    X(t) = X(0) + \sum_{k} \frac{l_k}{V} P_k(V\int^t_0 f_k(X(s))ds)
\end{equation}
where $P_k(t), \ k=\{1,\cdots, K\}$ are independent unit-rate Poisson
processes, and $l_{k} = c'_{k} - c_{k}\in \mathbb{R}^{d}$ denotes the
coefficient change of molecules at reaction $k$.

\subsection{Diffusion process} When $V$ is large, a
Poisson process can be approximated by a diffusion process. The follow
lemma in \cite{komlos1975approximation, komlos1976approximation} gives
the strong approximation theorem for Poisson processes. 

\begin{lemma}
\label{lem21}
A unit Poisson process $P(\cdot)$ and a Wiener process $B(\cdot)$ can be constructed so that
\begin{equation*}
    \left|\frac{P(Vt)-Vt}{\sqrt{V}}-\frac{1}{\sqrt{V}}B(Vt)\right|\leq \frac{\log(Vt\vee 2)}{\sqrt{V}}\Gamma,
\end{equation*}
where $\Gamma$ is a random variable such that
$\mathbb{E}(e^{c\Gamma})<\infty$ for some constant $c > 0$.
\end{lemma}

\begin{remark}
By the scaling property of Wiener process, $\frac{1}{\sqrt{V}}B(Vt)$
is also a standard Wiener process.
\end{remark}

With the lemma above and Ito's formula, we have the diffusion approximation
\begin{equation*}
\begin{aligned}
P_k\left(V\int^t_0 f_k(X(s))ds\right) &\approx V\int^t_0 f_k(X(s))ds  + \int_{0}^{t}\sqrt{V f_k(X(s))}dB(s)\\
&= V\int^t_0 f_k(X(s))ds + B_k\left(V\int^t_0 f_k(X(s))ds\right)
\end{aligned} 
\end{equation*}
This gives the diffusion approximation of the mass action system
$X(t)$:
\begin{equation*}
    Y(t) = Y(0)+\sum_k \frac{l_k}{V}\left[V \int^t_0 f_k(Y(s))ds + B_k\left(V\int^t_0 f_k(X(s))ds\right)\right].
\end{equation*}
In the chemistry literature, $Y$ is known as the Langevin
approximation for the continuous time Markov chain
model. Theoretically, the distance between these two approximations is
bounded as follow theorem in \cite{mozgunov2018review}.  
\begin{theorem}
\label{distance}
Let $X(t)$ be a Poisson process represented by \eqref{poss}, let $Y(t)$ be a diffusion process with initial condition satisfying $X(0)=Y(0)$ and solves the following stochastic differential equation
\begin{equation}
    \label{diff}
    Y(t) = Y(0)+\sum_k \frac{l_k}{V}\left[V \int^t_0 f_k(Y(s))ds + B_k\left(V \int^t_0 f_k(Y(s))ds\right)\right]
\end{equation}
where the $B_k(\cdot)$ are independent standard Wiener processes. As $V\rightarrow\infty$,
\begin{equation}
    \sup|X(t)-Y(t)| = O\left(\frac{\log V}{V}\right).
\end{equation}
\end{theorem}

The error of diffusion approximation is proportional to $\frac{\log V}{V}$, which converges to
0 as $V\rightarrow\infty$. In macroscopic chemical reaction system $V$
is at the magnitude of Avogadro's number. Therefore, the entire diffusion
term can be safely ignored. However, in many ecologic systems or
cellular chemical reaction systems, the effective volume cannot be
simply treated as infinity. This motivates us to consider the
sensitivity of the quasi-stationary-distributions (QSDs) against the
diffusion approximation. For any finite capacity $V$, the finite time
error of the diffusion approximation can be explicitly
simulated. Paper \cite{mozgunov2018review} gives the constructive 
procedure to generate discretized trajectories of the two processes
$X(t)$ and $Y(t)$ on the same probability space that they
stay close to each other \textit{trajectory by trajectory} with
probability one. We apply the algorithm to compute the finite time
error in section 3.   

\subsection{Coupling times}
In this paper, we use the coupling argument to connect finite time
error and the distance between QSDs. Let $\mu$ and $\nu$ be two
probability measures on a measurable space $(\mathcal{X}, \mathcal{B}(\mathcal{X}))$. A {\it coupling}
between $\mu$ and $\nu$ is a probability measure $\gamma$ on the product space
$(\mathcal{X}\times\mathcal{X}, \mathcal{B}(\mathcal{X})\times\mathcal{B}(\mathcal{X}))$ such that two marginal distribution of $\gamma$ are $\mu$
and $\nu$ respectively. 
\begin{definition}(Wasserstein distance) Let $d$ be a metric on the
  state space $S$. For probability measures $\mu$ and $\nu$ on $S$,
  the Wasserstein distance between $\mu$ and $\nu$ for $d$ is given by
\begin{equation}
\label{wass}
\begin{aligned}
    d_w(\mu,\nu) &=\inf\{\mathbb{E}_\gamma[d(x,y)]: \gamma \ \text{is a coupling of}\  \mu\  and\  \nu\}\\
    &=\inf \{\int d(x,y)\gamma(dx,dy): \gamma\  \text{is a coupling of} \ \mu \ and\  \nu\},
\end{aligned}
\end{equation}
\end{definition}
In this paper, without further specification, we assume that the
1-Wasserstein distance is induced by $d(x,y)= \min\{1, \|x-y\|\}$, where $\|x-y\|$ is the Euclidean norm.

Let $Z^{(1)}_{t}$ and $Z^{(2)}_{t}$ be two stochastic processes. A coupling
between $Z^{(1)}_{t}$ and $Z^{(2)}_{t}$ can be defined in the same way on the
space of paths. Throughout this paper, we assume $Z^{(1)}_{t + s} = Z^{(2)}_{t +
  s}$ for all $s > 0$ if $Z^{(1)}_{t} = Z^{(2)}_{t}$. In other words, $Z^{(1)}_{t}$ and
$Z^{(2)}_{t}$ stay together after their first meet. 

\begin{definition}(Coupling time)
The coupling time of a Markov coupling $(Z^{(1)}_t,Z^{(2)}_t)$ is a random variable given by
\begin{equation}
    \tau_c\overset{\text{def}}{=}\inf\{t\geq 0: Z_t^{(1)}=Z_t^{(2)}\}.
\end{equation}
\end{definition}

\begin{definition}(Successful coupling)
A coupling $(Z^{(1)}_t, Z^{(2)}_t)$ of Markov processes $Z^{(1)}$ and $Z^{(2)}$ is said to be successful if
\begin{equation}
    \mathbb{P}(\tau_c<\infty)=1.
\end{equation}
\end{definition}

We use the following reflection coupling to couple two diffusion
processes when they are far away from each other. 

\begin{definition}(Reflection coupling)
  Let $Z^{(1)}_{t}$ and $Z^{(2)}_{t}$ be two solutions of a stochastic
  differential equation
$$
  \mathrm{d}Z_{t} = f(Z_{t}) \mathrm{d} t + \sigma(Z_{t}) \mathrm{d}B_{t}
  $$
  when starting from different initial distributions. 
  A reflection coupling of $Z^{(1)}_{t}$ and $Z^{(2)}_{t}$ is made by reflecting
  the noise term about the orthogonal hyperplane at the midpoint
  between $Z^{(1)}_{t}$ and $Z^{(2)}_{t}$: 
\begin{equation}
\label{ref_coupling}
  \begin{aligned}
    \mathrm{d}Z^{(1)}_{t} &=  f(Z^{(1)}_{t}) \mathrm{d}t + \sigma(Z^{(1)}_{t})
    \mathrm{d}B_{t}\\
    \mathrm{d}Z^{(2)}_{t} &=  f(Z^{(2)}_{t}) \mathrm{d}t + \sigma(Z^{(2)}_{t}) (I-2\mathbf{e}\mathbf{e}^T)\mathrm{d}B_{t}
    \end{aligned}
\end{equation}
where $B$ is a standard Wiener process, and 
\begin{equation*}
    \mathbf{e}=\frac{1}{\|\sigma^{-1}(Z^{(1)}_{t}-Z^{(2)}_{t})\|}\sigma^{-1}(Z^{(1)}_{t}
    - Z^{(2)}_{t})
\end{equation*}
is a unite vector.
\end{definition}

We remark that the reflection coupling requires $\sigma(Z_t)$ in equation \eqref{ref_coupling} to be an invertible matrix. This is often not satisfied in the diffusion approximation \eqref{diff} because the number of Wiener processes in equation \eqref{diff} equals the number of reactions. Hence we often need to find a equivalent diffusion process with an invertible $\sigma$. See numerical examples for additional details. 

\medskip

The following maximal coupling is used to couple two processes that
are close to each other. 

\begin{definition}(Maximal coupling)
The maximal coupling looks for the maximal coupling probability for
the next step of $Z^{(1)}_t$ and $Z^{(2)}_t$. Assume
$Z^{(1)}_{t-1}$ and $Z^{(2)}_{t-1}$ are known and the
probability density function of $Z^{(1)}_t$ and
$Z^{(2)}_t$ is easy to compute. Following
\cite{jacob2020unbiased, johnson1998coupling}, the update of
$Z^{(1)}_t$ and $Z^{(2)}_t$ in Algorithm \ref{max} maximizes
the probability of coupling. 
\end{definition}

\begin{algorithm}
\caption{Maximal coupling}
\label{max}
\begin{algorithmic}[l]
  \REQUIRE $Z^{(1)}_{t-1}$ and $Z^{(2)}_{t-1}$ \\
\ENSURE $Z^{(1)}_t$ and $Z^{(2)}_t$, and $\tau_c$ if coupled
\STATE Compute probability density functions $p^{(1)}(z)$ and $p^{(2)}(z)$
\STATE Sample $Z^{(1)}_t$ and calculate $r=\mathcal{U}p^{(1)}(Z^{(1)}_t)$, where $\mathcal{U}$ is uniformly distributed on [0,1]
\IF{$r<p^{(2)}(Z^{(1)}_t)$}
\STATE $ Z^{(2)}_t=Z^{(1)}_t, \tau_c=t$
\ELSE
\STATE Sample $Z^{(2)}_t$ and calculate $r'=\mathcal{V} p^{(2)}(Z^{(2)}_t)$, where $\mathcal{V}$ is uniformly distributed on [0,1]
\WHILE{$r'< p^{(1)}(Z^{(2)}_t)$}
\STATE Resample $Z^{(2)}_t$ and $\mathcal{V}$. Recalculate $r'=\mathcal{V} p^{(2)}(Z^{(2)}_t)$
\ENDWHILE
\STATE $\tau_c$ is still undetermined
\ENDIF
\end{algorithmic}
\end{algorithm}

\subsection{Paired trajectories of Poisson process and of the diffusion process}
Recall that according to Lemma \ref{lem21} a unit-rate Poisson process has a strong diffusion approximation. Hence equation \eqref{poss} also has a strong approximation given by equation
\eqref{diff}. As the processes $P_k(\cdot)$ and
$B_k(\cdot)$ are continuous time processes, we apply the $\tau$-leaping
approximation for equation \eqref{poss} with the same step size $h$. This gives
\begin{equation}
  \label{possN}
    \hat{X}_{n+1} = \hat{X}_n + \sum_k
    \frac{l_k}{V}\left[P_k\left(Vh\sum^{n}_{m=0}f_k(\hat{X}_m)\right) -
    P_k\left(Vh\sum^{n-1}_{m=0}f_k(\hat{X}_m)\right)\right] 
\end{equation}
with $\hat{X}_0=X_0$. Similarly, the discretized
approximation of equation \eqref{diff} using the Euler-Maruyama method reads
\begin{equation}
\begin{aligned}
  \label{diffN}
   \hat{Y}_{n+1} &= \hat{Y}_n + \sum_k \frac{l_k}{V}(V h f_k(\hat{Y}_n))\\
   &+\sum_k
   \frac{l_k}{V}\left[B_k\left(Vh\sum^{n}_{m=0}f_k(\hat{Y}_m)\right) -
   B_k\left(Vh\sum^{n-1}_{m=0}f_k(\hat{Y}_m)\right)\right]  
\end{aligned}
\end{equation}
with initial condition $\hat{Y}_0=Y_0$.

The paired trajectories of $P_k(t)$ and $B_k(t)$ can be numerically generated by applying the KMT algorithm. The KMT algorithm actually generates a sequence of standard Poisson random variables $\{P_n\}$ and a sequence of standard normal random variables $\{W_n\}$, such that $\sum_{n = 1}^{N}P_n$ is approximated by $N + \sum_{n = 1}^{N} W_n$ for each finite $N$. Then after a re-scaling, one obtains a pair of discretized trajectories of $P_k(t)$ and $B_k(t)$ respectively. We refer \cite{mozgunov2018review} for a detailed review of the KMT algorithm. 

\section{Sensitivity of diffusion approximation}

\subsection{Quasi-stationary distribution}
Let $\mathbf{X} = \{X(t)\}$ and $\mathbf{\hat{X}}=\{\hat{X}_n\}$ (resp. $\mathbf{Y} =\{
Y(t)\}$ and $\mathbf{\hat{Y}}=\{\hat{Y}_n\}$) be the
stochastic process given by \eqref{poss} (resp. \eqref{diff}) and a
numerical approximation with step size $h$, respectively. Needless to
say a diffusion process is much easier to study than a Poisson process
with jumps. One natural question here is that how much the long time
dynamics of $\mathbf{X}$ is preserved by its diffusion
approximation. This problem is more complicated than it looks because both
$\mathbf{X}$ and $\mathbf{Y}$ have natural domain
$\mathbb{R}^{d}_{+}$. When the number of molecules of one species
reaches $0$, the process exits from its domain due to extinction. It
is common for equation \eqref{poss} or equation \eqref{diff} to have
finite time extinction. To
see this, consider the 1D version of equation \eqref{diff}:
\begin{equation}
  \label{1D}
  \mathrm{d}Y(t) = f( Y(t) ) \mathrm{d}t +
  \frac{1}{\sqrt{V}} \sqrt{f( Y(t))} \mathrm{d}B_{t} \,.
\end{equation}
Let $H(x) = x^{-1} $ be a test function. Applying Ito's formula then
take the expectation, we have
$$
  \frac{\mathrm{d}}{\mathrm{d}t} \mathbb{E}[ H(Y(t))] =
  f(Y(t))\left (\frac{2}{(Y(t))^{3}} -
  \frac{1}{(Y(t))^{2}} \right ) \,.
$$
If $f(Y(t)) = c Y(t) $ for a constant $c$, we have
$$
  \frac{\mathrm{d}}{\mathrm{d}t} \mathbb{E}[ H(Y(t))]  \geq 2c
  \left (\mathbb{E}[ H(Y(t))] \right )^{2} \,,
$$
which blows up to $\infty$ in finite time. Hence $Y(t)$ has
strictly positive extinction probability in finite time. The
calculation above fits the setting of many mass-action systems.

Therefore, to prevent finite time extinction, usually one needs constant influx of each species. That is why often we need to study the quasi-stationary distribution (QSD) instead of the invariant probability
distribution. Below we introduce the QSD and its sampling method only
for $\mathbf{X}$, as the case of $\mathbf{Y}$ is analogous. 

Let $\partial \mathcal{X} = \mathbb{R}^{d} \setminus
\mathbb{R}^{d}_{+}$ be the absorbing set of $\mathbf{X}$. The
quasi-stationary distribution (QSD) is an invariant probability measure
conditioning on $\mathbf{X}$ has not hit the absorbing set yet. We
further define
$$
  \tau = \inf\{t>0: X(t) \in\partial\mathcal{X} \}
$$
as the first passage time to $\partial \mathcal{X}$. 

\begin{definition}
A probability measure $\mu$ on $\mathbb{R}^{d}_{+}$ is called a
quasi-stationary distribution(QSD) of the Markov process $\mathbf{X}$ with
an absorbing set $\partial\mathcal{X}$, if for every measurable set $C\subset
\mathbb{R}^{d}_{+}$ 
\begin{equation}
    \mathbb{P}_{\mu}[ X(t) \in C|\tau>t]=\mu(C), \ t\geq 0,
    \label{1}
  \end{equation}
\end{definition}
\begin{definition}
If there is a probability measure $\mu$ exists such that
\begin{equation}
    \lim_{t\rightarrow\infty}\mathbb{P}_{x}[  X(t) \in
    C|\tau>t]=\mu(C),\ \forall x\in \mathbb{R}^{d}_{+} \,.
    \label{2}
\end{equation}
in which case we also say that $\mu$ is a quasi-limiting distribution(QLD).
\end{definition}

\begin{remark}
The limiting probability measure given by equation (\ref{2}), or the QLD, is also called the Yaglom limit. A QLD must be a QSD. Under some
mild assumptions about ergodicity, a QSD is also a QLD
\cite{darroch1965quasi}. 

If the first passage time of $\bm{X}$ to $\partial\mathcal{X}$ is $\infty$ with probability one, $\{\tau > t\}$ is the full probability space. As a result, QSD in equation \eqref{1} becomes the invariant probability measure and QLD in equation \eqref{2} becomes the limiting probability measure (which is also invariant). Therefore, when the mass action system admits an invariant probability measure instead of a QSD, all our arguments and algorithms still apply.

\end{remark}

When we define the numerical processes \eqref{possN} and
\eqref{diffN}, we need to specify the rule of regeneration such that
they both sample from QSDs as the time approaches to infinity. To
sample from QSD, we need to regenerate a sample once it hits the
absorbing set. Therefore, in addition to $\hat{X}_{n}$, we also need to update a temporal
occupation measure
$$
  \mu_{n} = \frac{1}{n}\sum_{k = 0}^{n-1} \delta_{\hat{X}_{k}} \,.
$$
If the numerical scheme gives $\hat{X}_{n+1} \in \partial
\mathcal{X}$, we immediately regenerate $\hat{X}_{n+1}$ from
$\mu_{n}$. More precisely,  let the transition kernel
of the numerical scheme of $\hat{X}_{n}$ (without resampling) be $\hat{Q}$. Then
$\hat{Q}$ has an absorbing set $\partial \mathcal{X}$ such that $\hat{Q}( \partial
\mathcal{X}, \partial \mathcal{X}) = 1$. The transition kernel of
$\hat{X}_{n}$ is the sum of $\hat{Q}$ and the regeneration
measure such that
$$
  \mathbb{P}[ \hat{X}_{n+1} \in A \,|\, \hat{X}_{n}
  = x] = \hat{Q}(x, A) + \hat{Q}(x, \partial \mathcal{X}) \mu_{n}(A) \,.
$$

The following convergence result follows from \cite{benaim2018stochastic}.

\begin{proposition}[Theorem 2.5 in \cite{benaim2018stochastic}]
  \label{prop1}
Let $\hat{\mu}$ be the QSD of the numerical process
$\hat{X}_n$. Under suitable assumptions about $\hat{X}_n$, the occupation measure
$\mu_{n}$ converges to the QSD $\hat{\mu}$ as $n \rightarrow \infty$.
\end{proposition}

To study the sensitivity of diffusion approximation, we also need a
theoretical process $\mathbf{\tilde{X}} = \{ \tilde{X}_{n} \}$ that directly regenerate
from the QSD $\hat{\mu}$ once exit to the boundary. Recall that
$\hat{Q}$ is the transition kernel of $\hat{X}_{n}$ (without
resampling). The transition kernel of $\mathbf{\tilde{X}}$ is
$$
  \tilde{P}(x, \cdot) = \hat{Q}(x, \cdot) + \hat{Q}(x, \partial
  \mathcal{X} ) \hat{\mu}  (\cdot)
  $$
  for all $x \in \mathbb{R}^{d}_{+}$. Note that $\hat{X}_{n}$ is
  not a Markov process (but $( \hat{X}_{n}, \mu_{n})$ is a
  Markov process). But $\mathbf{\tilde{X}}$ is a homogeneous Markov
  process with an invariant probability measure $\mu$. The case of
  $Y(t)$ is analogous. We denote the numerical process that
  resample from a temporal occupation measure by $\mathbf{\hat{Y}}=\{\hat{Y}_{n}\}$,
  and the Markov process that directly resample from QSD by
  $\mathbf{\tilde{Y}}=\{\tilde{Y_n}\}$.

\subsection{Decomposition of error term} 
Let $P_{X}$ and $\tilde{P}_{X}$ be the transition kernels of
$X(t)$ and $\tilde{X}_{n}$ respectively. Let $P_{Y}$ and
$\tilde{P}_{Y}$ be that of $Y(t)$ and $\tilde{Y}_{n}$
respectively. Denote the QSDs of $X(t)$, $\tilde{X}_{n}$,
$Y(t)$ and $\tilde{Y}_{n}$ by $\pi_{X}$, $\hat{\pi}_{X}$,
$\pi_{Y}$, and $\hat{\pi}_{Y}$, respectively. The quantity that we are
interested in is $d_{w}(\pi_{X}, \pi_{Y})$.

Let $T$ be a fixed constant.  Motivated by \cite{johndrow2017error},
the following decomposition follows easily by the triangle inequality
and the invariance.  
\begin{equation}
    d_w(\pi_X, \pi_Y) \leq d_w(\pi_X, \hat{\pi}_X) + d_w(\hat{\pi}_X, \hat{\pi}_Y) + d_w(\hat{\pi}_Y, \pi_Y) 
\end{equation}
The sensitivity of invariant probability against time discretization
has been addressed in \cite{dobson2021using}. When the time step size of 
the time discretization is small enough, the invariant probability
measure $\pi_Y$ is close to the numerical invariant probability
measure $\hat{\pi}_Y$. The case of QSD is analogous. Hence the third
term $d_w(\pi_Y, \hat{\pi}_Y)$ is proportional to step size $h$. The
estimation of the first term $d_w(\pi_X, \hat{\pi}_X)$ can be obtained
by some linear algebraic calculation.

\begin{theorem}
Let $X(t)$ be a continuous time Markov chain with finite state space and $\hat{X}$ be its tau-leaping approximation with step size $h$. Suppose that  $\pi$ and $\hat{\pi}$ be the true QSD and the numerical approximation of QSD respectively. Let $h$ be the time step size of numerical process. If the generating matrix of $X(t)$ is irreducible, then 
\begin{equation*}
    \|\pi-\hat{\pi}\|\sim O(h)
\end{equation*}
for $0 < h \ll 1$.
\end{theorem}

\begin{proof}
This proof follows the standard argument of eigenvector perturbation result. The case of stationary distribution is proved in \cite{meyer1994sensitivity}. Here we follow the argument in \cite{deutsch1985first} to prove a similar result for QSDs. Let $Q$ be the generating matrix of $X(t)$. Because $\pi$ is true QSD and $\hat{\pi}$ is the numerical approximation of QSD, we have
\begin{equation*}
    \pi^T e^{hQ} = \lambda \pi^T, \ \hat{\pi}^T (I+hQ) = \hat{\lambda}\hat{\pi}^T,
\end{equation*}
where $\lambda$ and $\hat{\lambda}$ are simple eigenvalues. Define a function
\begin{equation*}
A(t)\overset{\text{def}}{=}I+hQ+tR(h), 
\end{equation*}
where $R(h)$ is an $O(1)$ matrix given by the Taylor expansion $e^{hQ} = I+hQ+h^2 R(h)$. Then we have $A(0)=I+hQ$ and $A(h^2)=e^{hQ}$. Note that $A(0)$ is irreducible for all sufficiently small $h$ because $Q$ is also irreducible. Let $\pi(t)$ be the first eigenvector of $A(t)$ normalized to $1$ in $l_1$ norm. Then the sensitivity of $\pi$ is reduced to the derivative of $A(t)$.

Since $\pi$ is normalized to $1$ in $l_1$ norm, it follows from \cite{deutsch1985first} Section 3 that 
$$
\pi'(0) = S^{\sharp} A'(0) \pi(0)  \,,
$$
where $S = \lambda I - A(0)$, and $S^{\sharp}$ is the group inverse of $S$. (We refer \cite{deutsch1985first} for further discussion of the group inverse and derivative of Perron vector.) 

When $h$ is small, we have $1 - \lambda = O(h)$. Hence $S = I - O(h) - I - h Q$ is an $O(h)$ small matrix. This means $S^{\sharp} = O(h^{-1})$. In addition $A'(0) = R = O(1)$ by definition. Hence $\pi'(0) = O(h^{-1})$. Since $\hat{\pi} = \pi(h^2)$, we have 
$$
\| \pi - \hat{\pi} \| = O(h^{-1})\times O(h^2) = O(h) \,.
$$

This completes the proof. 

\end{proof}

Therefore, we have that $d_w(\pi_X, \hat{\pi}_X)=O(h)$ and $d_w(\pi_Y, \hat{\pi}_Y)=O(h)$, which make the second error term be the key part.
The second error term is the difference between numerical Poisson
process of a mass-action system and its corresponding numerical
diffusion process.

\begin{proposition}
Let $T>0$ be a fixed
constant. We can decompose $d_w(\hat{\pi}_X, \hat{\pi}_Y)$ via the
following inequality: 
\begin{equation}
    d_w(\hat{\pi}_X, \hat{\pi}_Y) \leq d_w(\hat{\pi}_X \tilde{P}^T_X, \hat{\pi}_X \tilde{P}^T_Y ) + d_w(\hat{\pi}_X\tilde{P}^T_Y, \hat{\pi}_Y\tilde{P}^T_Y)
\end{equation}
\end{proposition}

The term $d_w(\hat{\pi}_X \tilde{P}^T_X, \hat{\pi}_X \tilde{P}^T_Y)$ is
the finite time error and the term $d_w(\hat{\pi}_X\tilde{P}^T_Y,
\hat{\pi}_Y\tilde{P}^T_Y)$ can be bounded by coupling methods.  

There are two different ways to think about the distance
$d_w(\hat{\pi}_X, \hat{\pi}_Y)$. One method is considering
$\hat{\pi}_X$ and $\hat{\pi}_Y$ as conditional distributions on set
$\mathbb{R}^{d}_{+} / \partial\mathcal{X}$,
i.e. $\hat{\pi}_X(A)=\{\hat{X}\in A|t<\tau_X\}$ and
$\hat{\pi}_Y(A)=\{\hat{Y}\in A|t<\tau_Y\}$, where $\tau_X$ and
$\tau_Y$ are the killing time for processes $\mathbf{\hat{X}}$ and $\mathbf{\hat{Y}}$,
respectively. The other way is to use the $\mathbf{\tilde{X}}$ and
$\mathbf{\tilde{Y}}$ that regenerate from QSDs. No conditioning is
needed as $\hat{\mu}_{X}$ and $\hat{\mu}_{Y}$ are now the invariant probability measures of
$\mathbf{\tilde{X}}$ and $\mathbf{\tilde{Y}}$ respectively. There are
some fundamental difficulty when computing the conditional finite time error
because it is hard to couple $\hat{X}_{n}$ and
$\hat{Y}_{n}$ when one regenerates while the other does
not. Hence we choose to use $\mathbf{\tilde{X}}$ and
$\mathbf{\tilde{Y}}$ instead. 
\subsection{Finite time error}
We consider the modified processes $\mathbf{\tilde{X}}$ and $\mathbf{\tilde{Y}}$, which are
regenerated from the corresponding QSDs when they hit the boundary. Let
$\hat{\pi}_{X}$ and $\hat{\pi}_{Y}$ be the invariant
measures of $\mathbf{\tilde{X}}$ and $\mathbf{\tilde{Y}}$. Let $\tilde{\Gamma}(
\mathrm{d}x, \mathrm{d}y) =
\hat{\pi}^2_{X}(\tilde{P}^T_{X}\times\tilde{P}^T_{Y})$,
where $\hat{\pi}^2_{X}$ is the coupled measure of
$\hat{\pi}_{X}$ on the "diagonal" of
$\mathbb{R}^d\times\mathbb{R}^d$ that is supported by the hyperplane
$\{(x,y)\in\mathbb{R}^{2d}|y=x\}$ such that
$\hat{\pi}^2_{X}(\{(x,x)|x\in A\}) = \hat{\pi}_{X}(A)$,
and $\tilde{P}^T_{X}\times\tilde{P}^T_{Y}$ is any
coupled process such that two marginal processes are
$\mathbf{\tilde{X}}$ and $\mathbf{\tilde{Y}}$ respectively. The
following proposition follows easily. 

\begin{proposition}
Let $( \tilde{X}_{n}, \tilde{Y}_{n})$ be a coupling of
$\tilde{X}_{n}$ and $\tilde{Y}_{n}$ with transition kernel
$\tilde{P}^T_X\times\tilde{P}^T_Y$, then
\begin{equation*}
    d_w(\hat{\pi}_{X}\tilde{P}^T_{X},
    \hat{\pi}_{X}\tilde{P}^T_{Y})\leq
    \mathbb{E}_{\hat{\pi}_{X}}[d(\tilde{X}_{T},
    \tilde{Y}_{T})] \,.
  \end{equation*}
\end{proposition}

\begin{proof}
By the definition of Wasserstein distance
\begin{equation*}
    \begin{aligned}
    d_w(\hat{\pi}_{X}\tilde{P}^T_{X}, \hat{\pi}_{X}\tilde{P}^T_{Y}) &\leq\int_{\mathbb{R}^d\times\mathbb{R}^d} d(x,y)\hat{\pi}^2_{X}(\tilde{P}^T_{X}\times\tilde{P}^T_{Y})(dx, dy)\\
    &= \int_{\mathbb{R}^d} \mathbb{E}_{(x, x)}d( \tilde{X}_{T},
    \tilde{Y}_{T})\hat{\pi}_{X}(dx) = \mathbb{E}_{\hat{\pi}_{X}}[d(\tilde{X}_{T},
    \tilde{Y}_{T})] \,.
    \end{aligned}
\end{equation*}
\end{proof}

The key of estimating the finite time error effectively is to create a
good coupled process $( \tilde{X}_{n},
\tilde{Y}_{n})$. That is why we need to use the KMT algorithm to
generate matching Wiener process and Poisson processes. Here it
remains to define how $\tilde{X}_{n}$ and $\tilde{Y}_{n}$
couple when they regenerate from QSDs. Since we do not have QSD {\it
  in priori}, we will use $\hat{X}_{n}$ and $\hat{Y}_{n}$
to approximate $\tilde{X}_{n}$ and $\tilde{Y}_{n}$. In
other words, we regenerate samples from the temporal occupation
measure. To minimize error during sample regeneration, we define the
following coupled process $( \hat{X}_{n}, \mu^{X}_{n})$
and $( \hat{Y}_{n}, \mu^{Y}_{n})$, such that
$\hat{X}_{n}$ and $\hat{Y}_{n}$ follows equations
\eqref{possN} and \eqref{diffN} respectively by using paired processes
$B_{k}(t)$ and $P_{k}(t)$ for each $k$, and $\mu^{X}_{n}$,
$\mu^{Y}_{n}$ are two occupation measures. $S = (Z_{1}, \cdots,
Z_{N})$ ($N$ is large enough) is a finite sequence of uniform random variables on $(0, 1)$. Let $N_{X}$ and $N_{Y}$ are the total number of regenerations up to time $n$. In other words when $\tilde{X}_{n+1}$ enters $\partial
\mathcal{X}$ at step $n$ and needs regeneration, we increase $N_{X}$
by one and choose the $N_{X}$-th element of $S$, $Z_{N_{X}}$ to regenerate $\tilde{X}_{n+1}$,  by letting
$\tilde{X}_{n+1} = \tilde{X}_{\left \lfloor Z_{N_{X}}n
  \right \rfloor}$. Then it is easy to see that $( \hat{X}_{n}, \mu^{X}_{n})$
and $( \hat{Y}_{n}, \mu^{Y}_{n})$ is a Markov coupling
and the marginal processes $(\hat{X}_{n}, \hat{Y}_{n})$ is
a coupling of equations \eqref{possN} and \eqref{diffN}.

\begin{algorithm}
  \caption{Estimate finite time error}
  \begin{algorithmic}[l]
    \label{FTE_2}
    \REQUIRE Initial value $\hat{X}_0$
\ENSURE An estimator of $d_w(\hat{\pi}^T_{X}\tilde{P}^T_{X}, \hat{\pi}^T_{Y}\tilde{P}^T_{Y})$
\STATE Set initial value $\hat{X}^{1}_1 = \hat{Y}^{ 1}_1 $
\STATE Generate a sequence of uniformly distributed random variable S 
\FOR{$m = 1\ \text{to} \ M $}
\STATE Using the KMT algorithm to generate paired trajectories $\{P_k \}$ and $\{B_k\}$
\STATE If $m \neq 1$, reset initial value $\hat{X}^{m}_1 = \hat{Y}^{m}_1 = \hat{X}^{m-1}_T$
\STATE Let $N_X = N_Y = 0$
\FOR{$ n = 1\  \text{to}\  T$}
\STATE Update $\hat{X}^{m}_{n+1}$ and $\hat{Y}^{m}_{n+1}$ using equations \eqref{possN} and \eqref{diffN} respectively
\IF {$\hat{X}^{m}_{n+1} \in \partial \mathcal{X}$}
\STATE $N_X = N_X + 1$
\STATE Let $\hat{X}^{m}_{n+1} = \hat{X}^{m}_{\lfloor Z_{N_X}n\rfloor}$
\ENDIF
\IF {$\hat{Y}^{m}_{n+1} \in \partial \mathcal{X}$}
\STATE $N_Y = N_Y + 1$
\STATE Let $\hat{Y}^{m}_{n+1} = \hat{Y}^{m}_{\lfloor Z_{N_Y}n\rfloor}$
\ENDIF
\ENDFOR
\STATE Let $d(\hat{X}^{m}_T, \hat{Y}^{m}_T)=\min(1, \|\hat{X}^{m}_T - \hat{Y}^{ m}_T\|)$
\ENDFOR
\RETURN $\frac{1}{M}\sum^M_{m = 1} d(\hat{X}^{m}_T, \hat{Y}^{m}_T)$
\end{algorithmic}
\end{algorithm}

Details of computation are shown in Algorithm \ref{FTE_2}. When $N$ is
large, initial values $\hat{X}^{1}_1,\cdots,\hat{X}^{M}_1$ in Algorithm \ref{FTE_2} are from
a trajectory of the time-$T$ skeleton of $\hat{X}^T$. Hence
$\hat{X}^{1}_1, \hat{X}^{2}_1,\cdots,\hat{X}^{M}_1$ are approximately sampled from
$\hat{\pi}_X$. The error term $d(\hat{X}^{m}_T, \hat{Y}^{m}_T)$ evolved from
the initial value pair $\hat{X}^{m}_1=\hat{Y}^{m}_1=\hat{X}^{m-1}_T$ is recorded. Therefore, 
\begin{equation}
    \frac{1}{M}\sum^M_{m=1}d(\hat{X}^{m}_T, \hat{Y}^{m}_T)
\end{equation}
is an estimator of 
\begin{equation}
    \mathbb{E}_{\hat{\pi}_{X}}[d(\tilde{X}_{T},
    \tilde{Y}_{T})] \,,
\end{equation}
which is an upper bound of $d_w(\hat{\pi}^T_X\tilde{P}^T_X, \hat{\pi}^T_X\tilde{P}^T_Y)$.

\subsection{Coupling inequality and contraction rate}

Similar to the coupling inequality of the total variation norm, the distance $d$ we use in this paper also satisfies the coupling inequality. Let $(Z^{(1)}_t, Z^{(2)}_t)$ be a coupling of two stochastic processes and let $\tau_c$ be the coupling time. The following Lemma follows easily. 

\begin{proposition}
\label{coupling inequality}
For a Markov coupling $(Z^{(1)}_t, Z^{(2)}_t)$, we have
\begin{equation*}
    d_w(\mbox{law}(Z^{(1)}_T), \mbox{law}(Z^{(2)}_T ))\leq\mathbb{P}(Z^{(1)}_T\neq Z^{(2)}_T)=\mathbb{P}(\tau_c >T).
\end{equation*}
\end{proposition}
\begin{proof}
By the definition of the Wasserstein distance,
\begin{equation*}
    \begin{aligned}
    &d_w(\mbox{law}(Z^{(1)}_T), \mbox{law}(Z^{(2)}_T )) \leq \int d(\xi, \eta)\mathbb{P}((Z^{(1)}_T, Z^{(2)}_T)\in(d\xi, d\eta))\\
    &=\int_{\xi\neq\eta}d(\xi, \eta)\mathbb{P}((Z^{(1)}_T, Z^{(2)}_T)\in(d\xi, d\eta))\\
    &\leq \int_{\xi\neq\eta}\mathbb{P}((Z^{(1)}_T, Z^{(2)}_T)\in(d\xi, d\eta))\\
    &=\mathbb{P}(Z^{(1)}_T \neq Z^{(2)}_T).
    \end{aligned}
\end{equation*}
\end{proof}

\begin{proposition}
Assume that $d_w(\pi_X, \hat{\pi}_X)$ and $d_w(\pi_Y, \hat{\pi}_Y)$ are in order $O(h)$, then the error 
\begin{equation*}
  d_w(\pi_X, \pi_Y)\leq \frac{ d_w(\hat{\pi}_X \tilde{P}^T_X, \hat{\pi}_X \tilde{P}^T_Y )}{1-\alpha} + O(h),   
\end{equation*}
where $\alpha < 1$ is the contraction rate of the transition kernel $\tilde{P}^T_Y$ and $d_w(\hat{\pi}_X \tilde{P}^T_X, \hat{\pi}_X \tilde{P}^T_Y )$ is the finite time error.
\end{proposition}
\begin{proof}
By the triangle inequality,
\begin{equation*}
    d_w(\pi_X, \pi_Y) \leq d_w(\pi_X, \hat{\pi}_X) + d_w(\hat{\pi}_X, \hat{\pi}_Y) + d_w(\hat{\pi}_Y, \pi_Y). 
\end{equation*}
Because both $d_w(\pi_X, \hat{\pi}_X)$ and $d_w(\pi_Y, \hat{\pi}_Y)$ are $O(h)$, we only need to estimate the second term $d_w(\hat{\pi}_X, \hat{\pi}_Y)$. By the triangle inequality again, we have
\begin{equation*}
    d_w(\hat{\pi}_X, \hat{\pi}_Y) \leq d_w(\hat{\pi}_X \tilde{P}^T_X, \hat{\pi}_X \tilde{P}^T_Y ) + d_w(\hat{\pi}_X\tilde{P}^T_Y, \hat{\pi}_Y\tilde{P}^T_Y). 
\end{equation*}
If the transition kernel $\tilde{P}^T_Y$ has enough contraction such that
\begin{equation*}
    d_w(\hat{\pi}_X\tilde{P}^T_Y, \hat{\pi}_Y\tilde{P}^T_Y)\leq\alpha d_w(\hat{\pi}_X, \hat{\pi}_Y)
\end{equation*}
for some $\alpha<1$, then we have
\begin{equation}
\label{inequality}
    d_w(\hat{\pi}_X, \hat{\pi}_Y)\leq\frac{ d_w(\hat{\pi}_X \tilde{P}^T_X, \hat{\pi}_X \tilde{P}^T_Y )}{1-\alpha}.
\end{equation}
Therefore,
\begin{equation*}
  d_w(\pi_X, \pi_Y)\leq \frac{ d_w(\hat{\pi}_X \tilde{P}^T_X, \hat{\pi}_X \tilde{P}^T_Y )}{1-\alpha} + O(h),   
\end{equation*}
\end{proof}
Therefore, in order to estimate $d_w(\pi_X, \pi_Y)$, we need to look for suitable numerical estimators of the finite time error and the speed of contraction of $\tilde{P}^T_Y$. The finite time error can be easily estimated by Algorithm \ref{FTE_2}. And the speed of contraction $\alpha$ comes from the geometric ergodicity of the Markov process $\mathbf{\tilde{Y}}$ is approximated by that of $\mathbf{\hat{Y}}$ because of the convergence result in Proposition \ref{prop1}. If our numerical estimation gives
\begin{equation*}
 d_w(\hat{\pi}_X\tilde{P}^T_Y, \hat{\pi}_Y\tilde{P}^T_Y) \approx d_w(\hat{\pi}_X\hat{P}^T_Y, \hat{\pi}_Y\hat{P}^T_Y) \leq Ce^{-\gamma T},   
\end{equation*}
then we set $\alpha = e^{-\gamma T}$. Similar as in \cite{dobson2021using}, we use the following coupling method to estimate the contraction rate $\alpha$. Let $\hat{Z}=(\hat{Y}^{(1)}, \hat{Y}^{(2)})$ be a Markov process in $\mathbb{R}^{2d}$ such that $\hat{Y}^{(1)}$ and $\hat{Y}^{(2)}$ are two copies of $\hat{Y}$. Let the first passage time to the "diagonal" hyperplane $\{(\mathbf{x},\mathbf{y})\in\mathbb{R}^{2d}|\mathbf{y=x}\}$ be the coupling time. Then by Proposition \ref{coupling inequality}
\begin{equation*}
    d_w(\hat{\pi}_X\hat{P}^T_Y, \hat{\pi}_Y\hat{P}^T_Y)\leq\mathbb{P}(\tau_c>T).
\end{equation*}
As discussed in \cite{li2020numerical}, we need a hybrid coupling scheme to make sure that two numerical trajectories couple. Under the condition that two trajectories coupled before extinction time, some coupling methods such as reflection coupling or synchronous coupling are implemented in the first phase to bring two trajectories together. Then we compare the probability density function for the next step and couple these two numerical trajectories with the maximal possible probability (called maximal coupling). After doing this for many times, we have many samples of $\tau_c$ denote by $\bm{\tau}_c$. We use the exponential tail of $\mathbb{P}(\tau_c>t)$ to estimate the contraction rate $\alpha$. We look for a constant $\gamma>0$ such that
\begin{equation*}
    -\gamma = \lim_{t\rightarrow\infty}\frac{1}{t}\log(\mathbb{P}(\tau_c>t)
\end{equation*}
if the limit exists. See Algorithm 3 for the details of implementation of coupling. Note that we cannot simply compute the contraction rate start from $t=0$ because only the tail of coupling time can be considered as exponential distributed. In addition $\hat{Y}$ is a good approximation of $\tilde{Y}$ only if $t$ is large. Our approach is to check the exponential tail in a log-linear plot. After having $\bm{\tau}_c$ , it is easy to choose
a sequence of times $t_0, t_1,\cdots, t_n$ and calculate $n_i = |\{\tau^m_c > t_i| 0\leq m \leq M\}|$ for each
$i = 0,\cdots,n$. Then $p_i = n_i/M$ is an estimator of $\mathbb{P}_{\hat{\pi}_Y}[\tau_c>t_i]$. Now let $p^u_i$ (resp. $p^l_i$) be the
upper (resp. lower) bound of the confidence interval of $p_i$ such that
\begin{equation*}
    p^u_i = \tilde{p} + z\sqrt{\frac{\tilde{p}}{\tilde{n}_i}(1-\tilde{p})}\ (resp. \ p^l_i = \tilde{p} - z\sqrt{\frac{\tilde{p}}{\tilde{n}_i}(1-\tilde{p})}),
\end{equation*}
where $z=1.96$, $\tilde{n}_i=n_i+z^2$ and $\tilde{p}=\frac{1}{\tilde{n}}(n_i+\frac{z^2}{2})$ \cite{agresti1998approximate}. If $p^l_i\leq e^{-\gamma t_i}\leq p^u_i$ for each $0\leq i\leq n$,
we say that the exponential tail starts at $t=t_{i_0}$. we
accept the exponential tail with rate $e^{-\gamma T}$ if the confidence interval $p^u_{i_0}-p^l_{i_0}$ is sufficient small. Otherwise we need to run Algorithm \ref{contraction} for longer time to eliminate the initial bias in $\tau_c$.

\begin{algorithm}
\caption{Estimation of contraction rate $\alpha$}
\begin{algorithmic}[l]
 \label{contraction}
  \REQUIRE Initial values $x,y\in \mathcal{X}/\partial{\mathcal{X}}$ \\
\ENSURE An estimation of contraction rate $\alpha$ 
\STATE Choose threshold $d>0$
\FOR{$m=1\  \text{to}\  M$}
\STATE  $\tau^m_c=0, t=0, (\hat{Y}^{(1)}_0, \hat{Y}^{(2)}_0)=(x, y)$
\STATE Flag = 0
\WHILE{Flag=0}
\IF{$\hat{Y}^{(1)}_t$ and $\hat{Y}^{(2)}_t\in\mathcal{X}/\partial{\mathcal{X}}$}
\IF{$|\hat{Y}^{(1)}_t-\hat{Y}^{(2)}_t|>d$}
\STATE Compute $(\hat{Y}^{(1)}_{t+1}, \hat{Y}^{(2)}_{t+1})$ using reflection
coupling or independent coupling
\STATE $t\leftarrow t+1$
\ELSE
\STATE Compute $(\hat{Y}^{(1)}_{t+1}, \hat{Y}^{(2)}_{t+1})$ using maximal coupling
\IF{coupled successfully}
\STATE Flag=1
\STATE $\tau^m_c=t$
\ELSE
\STATE $t\leftarrow t+1$
\ENDIF
\ENDIF
\ENDIF
\ENDWHILE
\ENDFOR
\STATE Use $\tau^1_c,\cdots,\tau^M_c$ to compute $\mathbb{P}(\tau_c>t|t<\min(\tau_{Y^{(1)}},\tau_{Y^{(2)}}))$
\STATE Fit the tail of $\log\mathbb{P}(\tau_c>t|t<\min(\tau_{Y^{(1)}},\tau_{Y^{(2)}}))$ versus $t$ by linear regression. Compute the slope $\gamma$.
\end{algorithmic}
\end{algorithm}

\section{Numerical Examples}
\subsection{SIR model}Consider an epidemic model in which the whole population is divided into three distinct classes S(susceptible), I(infected) and R(recovered), respectively. After non-dimensionalization, the ODE version of an SIR model reads
\begin{equation}
\label{x_numerical}
    \begin{aligned}
    \frac{dS}{dt} &=(\alpha - \beta S I -\mu S)\\
    \frac{dI}{dt} &=(\beta S I -(\mu+\rho+\gamma)I)\\
    \frac{dR}{dt} &=(\gamma I - \mu R)
    \end{aligned}
\end{equation}
where $\alpha$ is the birth rate, $\mu$ is 
the disease-free death rate,$\rho$ is the excess death rate for the infected class,$\gamma$ is the recover rate for the infected population,and $\beta$ is the effective contact
rate between the susceptible class and infected class \cite{dieu2016classification}. Note R completely depends on $S$ and $I$. So we just consider the evolutions of S and I. 

Now we let $V$ be the total population and consider the corresponding stochastic mass action network. There are four reactions are involved in this network. The stochastic mass action network can be defined by a Poisson process $X_n = (S_n, I_n)$.
\begin{equation}
\label{y_numerical}
    \begin{aligned}
    & \emptyset \overset{\alpha} {\Rightarrow}S, \ 
    S + I \overset{\beta}{\Rightarrow} 2I\\
    & S \overset{\mu}{\Rightarrow} \emptyset, \ 
    I \overset{\mu+\rho+\gamma}{\Rightarrow} \emptyset
    \end{aligned}
\end{equation}

Applying the numerical representation in \eqref{possN}, we have the approximate rate functions of Poisson process $\hat{X}_n$:
\begin{equation*}
\begin{aligned}
    &q_{1,n} = \sum^{n-1}_{m=0} V h \alpha, \  q_{2,n} = \sum^{n-1}_{m=0}V h \beta S_m I_m, \\ 
    &q_{3,n} =\sum^{n-1}_{m=0}V h \mu S_m, \
    q_{4,n} =\sum^{n-1}_{m=0} V h (\mu+\rho+\gamma) I_m. 
\end{aligned}
\end{equation*}

Let $P_i, i = {1, 2, 3, 4}$ be independent unit rate Poisson processes. Then $\hat{X}_n$ is driven by the discrete approximation of $\{ P_i \}_{i = 1}^4$. The rule of update of the numerical approximation $\hat{X}_n$ follows 
\begin{equation}
\label{SIR_X}
    \hat{X}_{n+1} =\begin{pmatrix}
    S_{n+1}\\
    I_{n+1}
    \end{pmatrix} 
    = \begin{pmatrix}
    S_n\\
    I_n
    \end{pmatrix} + \frac{1}{V}\begin{pmatrix}
   {\bm f}_1(P_1, \cdots, P_4, q_{1,n}, \cdots , q_{4,n})\\
    {\bm f}_2(P_1, \cdots, P_4, q_{1,n}, \cdots, q_{4,n}) 
    \end{pmatrix},
\end{equation}
where ${\bm f}_1$ and ${\bm f}_2$ comes from discrete approximation in equation \eqref{possN}. To improve the readability of the present paper, we move detailed expressions of ${\bm f}_1$ and ${\bm f}_2$ to the appendix. 

As described in Section 2.1, each Poisson processes $P_i, i = 1, 2, 3, 4$ is path-wisely approximated by a Wiener process $B_i, i = 1, 2, 3, 4$. Further, the discrete approximation $\hat{X}_n$ is pathwisely approximated by a Euler-Maruyama scheme $\hat{Y}_n$ reads
\begin{equation}
\label{SIR_Y1}
    \hat{Y}_{n+1} =\begin{pmatrix}
    S_{n+1}\\
    I_{n+1}
    \end{pmatrix} = \begin{pmatrix}
    S_n\\
    I_n
    \end{pmatrix} + \frac{1}{V}\begin{pmatrix}
    {\bm g}_1 (q_{1,n}, \cdots , q_{4,n})\\
    {\bm g}_2 (q_{1,n}, \cdots , q_{4,n})
    \end{pmatrix} 
     + \frac{1}{V}\begin{pmatrix}
    {\bm \sigma}_1(  B_1, \cdots, B_4, q_{1,n}, \cdots , q_{4,n})\\
    {\bm \sigma}_2(B_1, \cdots, B_4, q_{1,n}, \cdots , q_{4,n})
    \end{pmatrix},
\end{equation}
where functions ${\bm g}_1, {\bm g}_2, {\bm \sigma }_1$, and ${\bm \sigma}_2$ follows the expression in equation \eqref{diffN}. We refer the appendix for the detailed form of these functions. 

By the stationary increments property of standard Wiener process, we know that every finite difference of $B_i$ is normally distributed. In addition Wiener processes $B_i, i = 1, 2, 3, 4$ are independent. Therefore, equation \eqref{SIR_Y1} can be simplified to: 
\begin{equation}
\label{SIR_diff_matrix}
    \hat{Y}_{n+1} =\begin{pmatrix}
    S_{n+1}\\
    I_{n+1}
    \end{pmatrix} = \begin{pmatrix}
    S_n\\
    I_n
    \end{pmatrix} + \frac{1}{V}\begin{pmatrix}
    {\bm g}_1 (q_{1,n}, \cdots , q_{4,n})\\
    {\bm g}_2 (q_{1,n}, \cdots , q_{4,n})
    \end{pmatrix} 
    + \frac{1}{V} M \begin{pmatrix}
     W_1 \\ W_2 \\ W_3 \\ W_4
    \end{pmatrix}
\end{equation}
where $W_i, i = 1, \cdots, 4$ are independent standard normal random variables, and $M$ is a matrix that depends only on $S_n$ and $I_n$. We refer readers to the appendix for the full expression of $M$. 

In order to estimate the distance between two QSDs, we need to find the contraction rate $\alpha$ for diffusion process $\hat{Y}$ above. However, the diffusion matrix $M$ in $\hat{Y}$ is not square, which makes a reflection coupling  difficult. Here we define an equivalent diffusion process that is driven by a 2D Wiener process but has the same law as $\hat{Y}$. In our simulation, we compute the 2 by 2 covariance matrix $N=MM^T$, and set the square root of $N$ to be the new diffusion matrix. Then $\hat{Y}$ can be re-written as
    \begin{equation}
    \begin{aligned}
    \label{SIR_Y2}
    \hat{Y}_{n+1} &=\begin{pmatrix}
    S_n\\
    I_n
    \end{pmatrix} + \frac{1}{V}\begin{pmatrix}
    {\bm g}_1 (q_{1,n}, \cdots , q_{4,n})\\
    {\bm g}_2 (q_{1,n}, \cdots , q_{4,n})
    \end{pmatrix} \\ &+ \frac{1}{\sqrt{\operatorname{tr}(N)+2\sqrt{\operatorname{det}(N)}}}(N+\operatorname{det}(N) Id)\begin{pmatrix}
     W_1\\
     W_2
    \end{pmatrix},
    \end{aligned}
    \end{equation}
where $\operatorname{tr}(N)$ is the trace of $N$ and $\operatorname{det}(N)$ is the determinant of $N$, and $Id$ is the identity matrix. It is easy to see that the diffusion process $\hat{Y}$ in equations \eqref{SIR_diff_matrix} and \eqref{SIR_Y2} are equivalent. Hence we do not change its notation here. The modification of $\hat{Y}$ allows us to run Algorithm \ref{contraction} to compute the coupling time distribution. 

It remains to compute the finite time error. Let $\partial\mathcal{X}$ be the union of x-axis and y-axis. The model parameters are set as $\alpha = 7, \beta = 3, \mu = 1, \rho = 1, \gamma=2$. Processes $\hat{X}$ and $\hat{Y}$ admit QSDs $\hat{\pi}_X$ and $\hat{\pi}_Y$, respectively.
Long trajectories $P(i\Delta)$ and $B(i\Delta)$ for $i=\{1,\cdots, 2^{20}\}$ and $\Delta=0.01$ are constructed when we consider the trajectory-by-trajectory behaviour of two processes. The time step size is $h=0.001$ and the fixed time is set as $T=0.5$.

The result for $V=1000$ is demonstrated in Figure 1. Left bottom of Figure 1 shows the QSD of diffusion process $\hat{Y}$. The QSD of the Poisson process is shown on right top of Figure 1. The difference of these two QSDs is shown at the bottom of Figure 1. We can see that the total variation distance between two QSDs is $0.0901$, which is considered to be small. This is reasonable because with high probability, the trajectories of both Poisson process and the diffusion process moves far away from the absorbing set $\partial\mathcal{X}$. 

The total variation distance between two QSDs is consistent with the prediction developed in this paper. We first use Algorithm 3 to compute the distribution of the coupling time, which is shown in Figure 1 Top Left. Then we use Algorithm 2 to compute the finite time error. The finite time error is 0.0026 for $V = 1000$. As a result, the upper bound given in equation \eqref{inequality} is 0.0054 for $V = 1000$, which is smaller than the empirical total variation error 0.0901 in this case.

Then we carry out similar computations for $V=10$ on a course mesh. The result is shown in Figure 2. To compare with the case for $V=1000$ on the same mesh, we re-scaled the probability density function obtained from the Monte-Carlo simulation. The probability density in one bin in the coarse mesh is evenly distributed into many bins in the refined mesh. The difference between two QSDs are shown at the bottom of Figure 2. It is not hard to see the total variation distance becomes significantly larger when the volume gets smaller. Same as above, we use Algorithm 3 to compute the distribution of the coupling time distribution ( Figure 2 Top Left) and use Algorithm 2 to compute the finite time error. The finite time error is 0.1748 for $V=10$. As a result, the upper bound given in \eqref{inequality} is 0.3639 for $V=10$. This is consistent with the numerical finding shown in Figure 2 Bottom Right.

\begin{figure}[h]
 \centering
     \includegraphics[width=\textwidth]{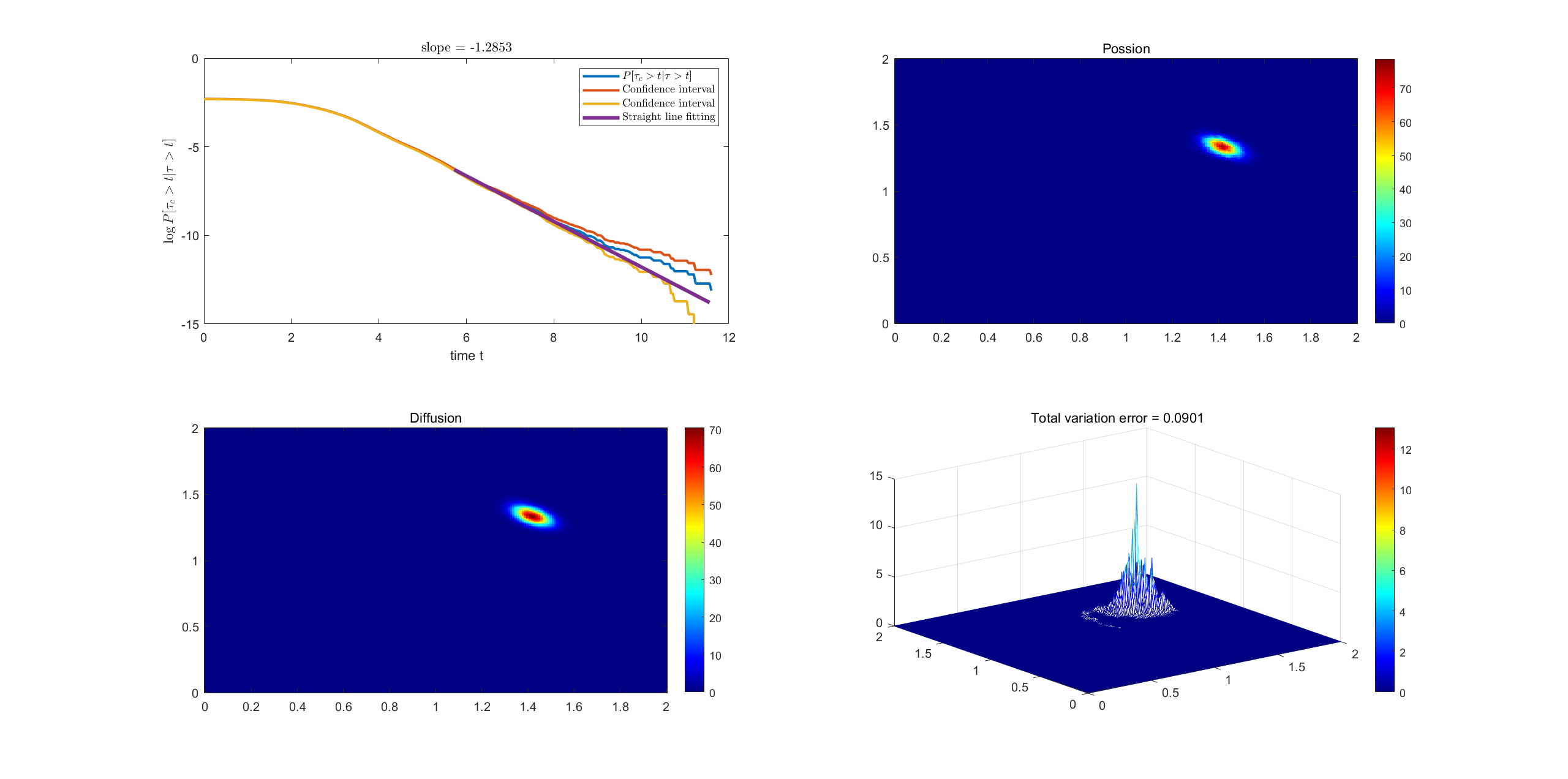}
     \caption{(Case $V = 1000$) $\textbf{Upper panel}$: ($\textbf{Left}$) $\mathbb{P}(\tau_c>t|\tau<t)$ vs.$t$. ($\textbf{Right}$) QSD of Poisson process.\ $\textbf{Lower panel}$: ($\textbf{Left}$) QSD of diffusion process. ($\textbf{Right}$) Total variation of two QSDs.}     
    \label{F1}
 \end{figure}
 
\begin{figure}[h]
 \centering
     \includegraphics[width=\textwidth]{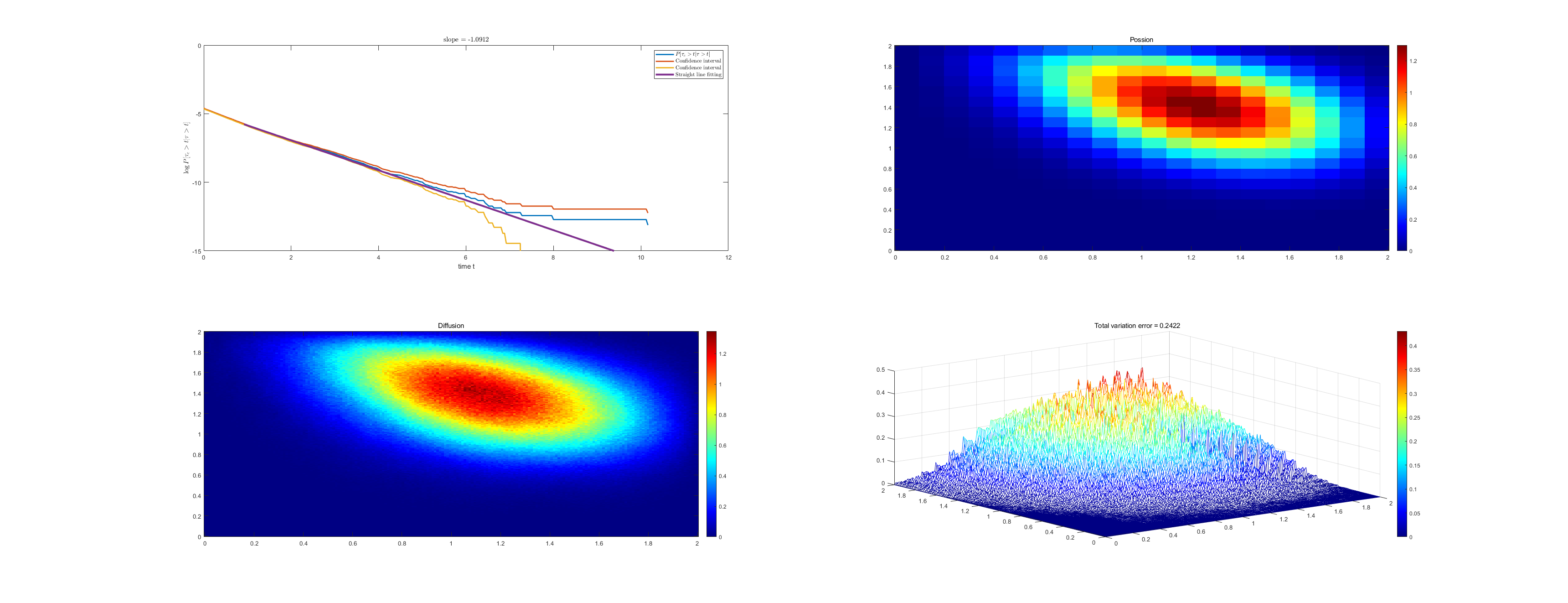}
     \caption{(Case $V = 10$) $\textbf{Upper panel}$: ($\textbf{Left}$) $\mathbb{P}(\tau_c>t|\tau<t)$ vs.$t$. ($\textbf{Right}$) QSD of Poisson process.\ $\textbf{Lower panel}$: ($\textbf{Left}$) QSD of diffusion process. ($\textbf{Right}$) Total variation of two QSDs.}     
    \label{F2}
 \end{figure} 

As we consider the effect of the capacity volume, the finite time error and the contraction rate for different volumes are compared in Table \ref{t1}. The last column $d_w(\hat{\pi}_X, \hat{\pi}_Y)$ is computed using \eqref{inequality}. Being consistent with Theorem \ref{distance}, the 1-Wasserstein distance between two QSDs is smaller as $V$ getting larger. 

\begin{center}
\begin{table}
\begin{tabular}{|c|c| c| c|c| }
\hline
 volume $V$ &finite time error & contraction rate $\gamma$ & $d_w(\hat{\pi}_X, \hat{\pi}_Y)$ \\
 \hline
 1000 &0.0026 & 1.2853 & 0.0054 \\  
 400 & 0.0079 & 1.2418 & 0.0170\\
 100 &0.0279 & 1.1613  & 0.0634\\
 10  &0.1748 & 1.0912 & 0.3639\\
 \hline
\end{tabular} 
\captionsetup{justification=centering}
\caption{SIR model. Numerical results for different volumes}
\label{t1}
\end{table}
\end{center}

\subsection{Oregonator system}
In this example, we consider a well known example of chemical oscillator called the  Belousov-Zhabotinsky (BZ) reaction model or "Oregonator"\cite{brons1991canard, epstein1998introduction, gray1990chemical}. The ODE version of an Oregnator system is given by
\begin{equation*}
    \begin{aligned}
    \frac{dS_1}{dt} &= S_1S_2 - C_2S_1S_2 + C_3S_1-2C_4S_1^2\\
    \frac{dS_2}{dt} &=-C_1S_2-C_2S_1S_2+C_5hS_3\\
    \frac{dS_3}{dt} &=2C_3S_1-C_5S_3.
    \end{aligned}
\end{equation*}
We refer Figure 3 Top Left for a sample trajectories of the Oregonator on $\mathbb{R}^3_{+}$. The parameter values are chosen as $C_1=2560, C_2=800000, C_3=16000, C_4=2000, C_5=9000, \delta=0.4$.

Let $V$ be the volume. Six reactions in this process are shown as following.
\begin{equation*}
\begin{aligned}
 & S_2 \overset{C_1}{\Rightarrow} S_1, \ 
 S_1+S_2 \overset{C_2}{\Rightarrow} \emptyset, \ 
 S_1 \overset{C_3}{\Rightarrow}2S_1+2S_3\\
 & 2S_1\overset{C_4}{\Rightarrow} \emptyset,\ 
 S_3\overset{C_5\delta}{\Rightarrow}S_2, \ 
 S_3\overset{C_5(1-\delta)}{\Rightarrow} \emptyset
\end{aligned}
\end{equation*}
Applying the numerical representation in \eqref{possN}, we have the approximate rate functions
of Poisson process $\hat{X}_n = (S_{1,n}, S_{2,n}, S_{3,n})$:
\begin{equation*}
\begin{aligned}
    & q_{1,n}=
    \sum^{n-1}_{m=0}V h C_1 S_{2,m},\ q_{2,n}=\sum^{n-1}_{m=0}V h C_2 S_{1,m}S_{2,m}, \ q_{3,n}=\sum^{n-1}_{m=0}V h C_3 S_{1,m}, \\
    & q_{4,n}=\sum^{n-1}_{m=0}V h C_4S_{1,m}^2, \ q_{5,n}=\sum^{n-1}_{m=0} V h C_5 \delta S_{3,m}, \ q_{6,n}=\sum^{n-1}_{m=0} V h C_5 (1-\delta) S_{3,m}.
\end{aligned}
\end{equation*}
We remark terms $S_{1,m}$ is the numerical value of species $S_1$ at time step $m$, and cases of other terms are analogous.
Hence the Poisson process $\hat{X}$ of the Oregonator model can be written as
\begin{equation*}
 \hat{X}_{n+1} =\begin{pmatrix}
    S_{1,n+1}\\
    S_{2, n+1}\\
    S_{3, n+1}
    \end{pmatrix} =\begin{pmatrix}
    S_{1, n}\\
    S_{2, n}\\
    S_{3, n}
    \end{pmatrix}+ \frac{1}{V}\begin{pmatrix}
    \bm{f_1}(P_1,\cdots, P_6, q_{1, n},\cdots, q_{6,n})\\
    \bm{f_2}(P_1,\cdots, P_6, q_{1, n},\cdots, q_{6,n})\\
    \bm{f_3}(P_1,\cdots, P_6, q_{1, n},\cdots, q_{6,n})
    \end{pmatrix},
\end{equation*}
where $P_i, i=\{1,\cdots, 6\}$ are independent unite rate Poisson processes. $\bm{f_1}$, $\bm{f_2}$ and $\bm{f_3}$ comes from discrete approximation in equation $\eqref{possN}$. To improve the
readability of the present paper, we move detailed expressions of $\bm{f_1}$, $\bm{f_2}$ and $\bm{f_3}$ to the appendix.

The diffusion approximation $\hat{Y}$ can be written as
\begin{equation}
\label{OREG_Y1}
  \hat{Y}_{n+1} = \begin{pmatrix}
    S_{1,n+1}\\
    S_{2,n+1}\\
    S_{3,n+1}
    \end{pmatrix} \\
    =\begin{pmatrix}
    S_{1,n}\\
    S_{2,n}\\
    S_{3,n}
    \end{pmatrix} + \frac{1}{V}\begin{pmatrix}
     \bm{g_1}(q_{1,n}, \cdots, q_{6,n})\\
     \bm{g_2}(q_{1, n}, \cdots, q_{6, n})\\
     \bm{g_3}(q_{1, n}, \cdots, q_{6, n})
    \end{pmatrix} +
   \frac{1}{V}\begin{pmatrix} \bm{\sigma_1}(B_1,\cdots, B_6, q_{1, n}, \cdots, q_{6, n})\\
    \bm{\sigma_2}(B_1,\cdots, B_6, q_{1, n}, \cdots, q_{6, n})\\
     \bm{\sigma_3}(B_1,\cdots, B_6, q_{1, n}, \cdots, q_{6, n})
   \end{pmatrix} 
\end{equation}
where $B_i, i= \{1, \cdots, 6\}$ are independent standard Wiener processes, functions $\bm{g_1}$, $\bm{g_2}$, $\bm{g_3}$, $\bm{\sigma_1}$, $\bm{\sigma_2}$ and $\bm{\sigma_3}$ follows the expression in equation $\eqref{diffN}$. We refer the
appendix for the detailed form of these functions.

By the stationary increments property and independence of Wiener processes $B_i, i= \{1, \cdots, 6\}$, equation \eqref{OREG_Y1} can be simplified to: 
\begin{equation}
\label{OREG_diff_matrix}
\begin{aligned}
    \hat{Y}_{n+1} &=\begin{pmatrix}
    S_{1,n+1}\\
    S_{2,n+1}\\
    S_{3,n+1}
    \end{pmatrix} = \begin{pmatrix}
    S_{1,n}\\
    S_{2,n}\\
    S_{3,n}
    \end{pmatrix} + \frac{1}{V}\begin{pmatrix}
   \bm{g_1}(q_{1,n}, \cdots, q_{6,n})\\
     \bm{g_2}(q_{1, n}, \cdots, q_{6, n})\\
     \bm{g_3}(q_{1, n}, \cdots, q_{6, n})
    \end{pmatrix} 
    + \frac{1}{V} M \begin{pmatrix}
     W_1 \\ W_2 \\ W_3 \\ W_4 \\W_5 \\ W_6
    \end{pmatrix}
\end{aligned}
\end{equation}
where $W_i, i = 1, \cdots, 6$ are independent standard normal random variables, and $M$ is a matrix that depends only on $S_n$ and $I_n$. We refer readers to the appendix for the full expression of $M$. 

Let $\partial\mathcal{X}$ be union of x-axis, y-axis and z-axis. Processes $\hat{X}$ and $\hat{Y}$ admit QSDs $\hat{\pi}_X$ and $\hat{\pi}_Y$, respectively.
Long trajectories $P(i\Delta)$ and $B(i\Delta)$ for $i=\{1,\cdots,
2^{29}\}$ and $\Delta=0.001$ are constructed when we consider the
trajectory-by-trajectory behaviour of two processes. The time step
size is $h=10^{-8}$ and the fixed time is set as $T=2\times 10^{-4}$
when $V = 1000$, $T = 4\times 10^{-5}$ when $V = 400$, $T = 1\times
10^{-5}$ when $V = 100$, and $T = 2\times 10^{-6}$ when $V = 10$. Note
that large rate coefficients $C_i$ make the numerical results easily to beyond the length of long trajectory $B(i\Delta)$, so we pick small time step size $h$ and the fixed finite time $T$.

Figure 3 Top Left shows the solution of the ordinary differential equation. For any initial point, the trajectory eventually converges to the limit cycle. In terms of thermodynamics, the oscillation is induced through dissipation of energy and is often called a self-sustained oscillator\cite{miyazaki2013pattern}. The trajectories of Poisson process and the diffusion process up to fixed time $T$ are shown on the Top Right and Bottom Left. It looks that the trajectories are close and this is reasonable because with high probability, the trajectories of both Poisson process and the diffusion process moves far away from the absorbing set $\partial\mathcal{X}$. There are only a few regeneration events (the lines crossing the limit cycle). We compute the distribution of the coupling time. The coupling time distribution and its exponential tail are shown in Figure 3 Top Left. Then we use Algorithm \ref{FTE_2} to compute the finite time error. The finite time error is 0.0057 for $V = 1000$. As a result, the upper bound given in \eqref{inequality} is 0.0116 for $V = 1000$. For $V=10$, the finite time error is 0.4531 and the upper bound given in \eqref{inequality} is 0.4531. 

To compare the different situations for volume $V=1000$ and $V=10$, we plot the trajectories for both processes for each species. Trajectories for $V=1000$ is shown in the upper row of Figure 4 and lower row shows the case for $V=10$. It is not hard to see the Poisson process is quite close to the diffusion process when $V=1000$. But when the volume is too small, not much Poisson jumps can be observes in the Poisson process, while significant noise can be seen in the diffusion approximation. As a result, the finite time error for $V=10$ is 0.0563, which is around ten times larger than that for $V=1000$. Same as above, we compute the contraction rate $\gamma$ of the coupling time distribution to be $2.0927 \times 10^5$. This is due to the large magnitude of noise in the diffusion approximation. As a result, the upper bound given in \eqref{inequality} is 0.4531 for $V=10$. We conclude that the diffusion approximation does not approximate the QSD well when the volume is not large enough. 

\begin{figure}[h]
 \centering
     \includegraphics[width=\textwidth]{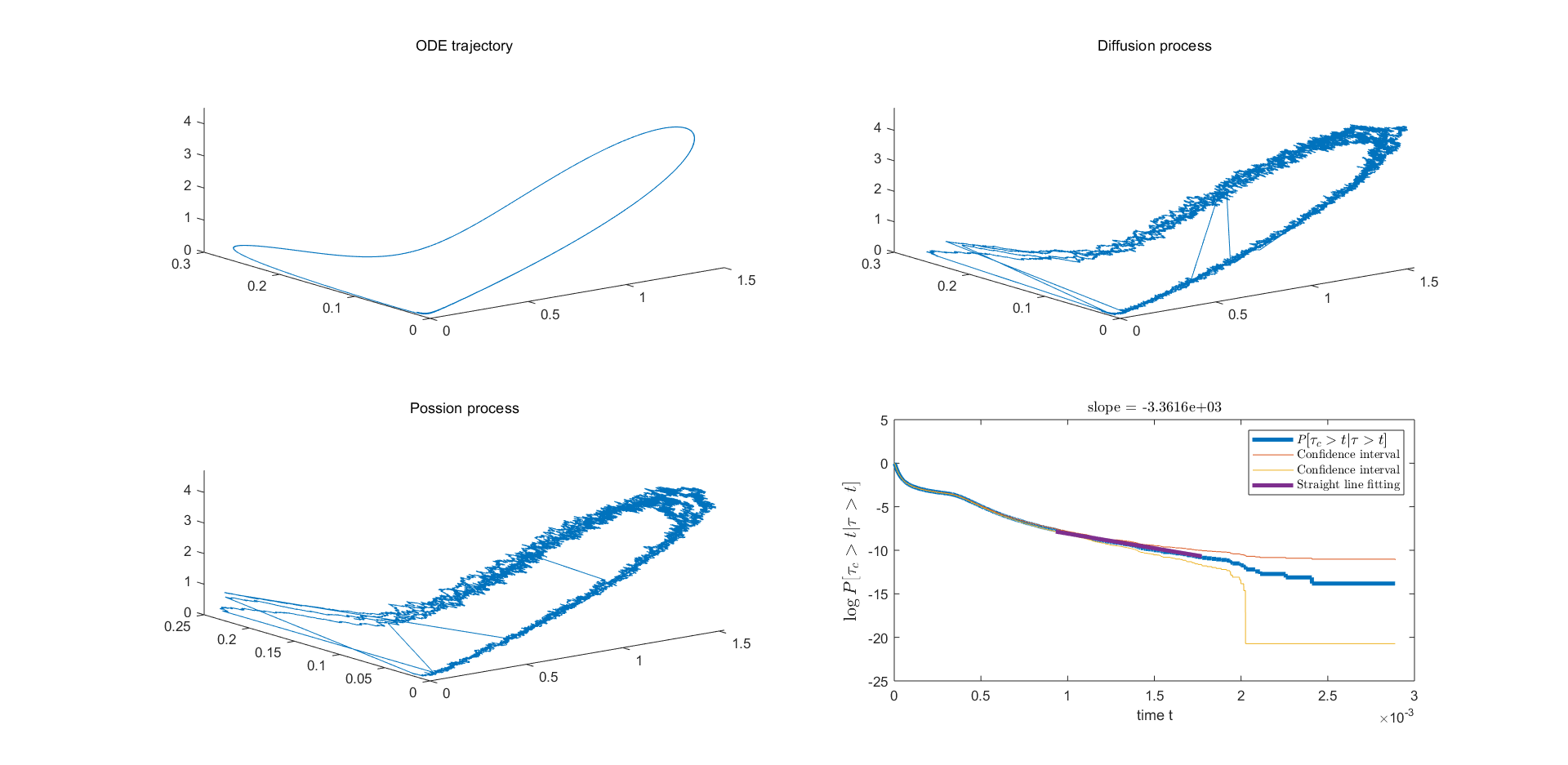}
     \caption{(Case $V = 1000$) $\textbf{Upper panel}$: ($\textbf{Left}$) ODE trajectories. ($\textbf{Right}$) Trajectories of Poisson process.\ $\textbf{Lower panel}$: ($\textbf{Left}$) Trajectories of diffusion process. ($\textbf{Right}$) $\mathbb{P}(\tau_c>t|\tau<t)$ vs.$t$.}     
    \label{F1}
 \end{figure}
 
\begin{figure}[h]
 \centering
     \includegraphics[width=\textwidth]{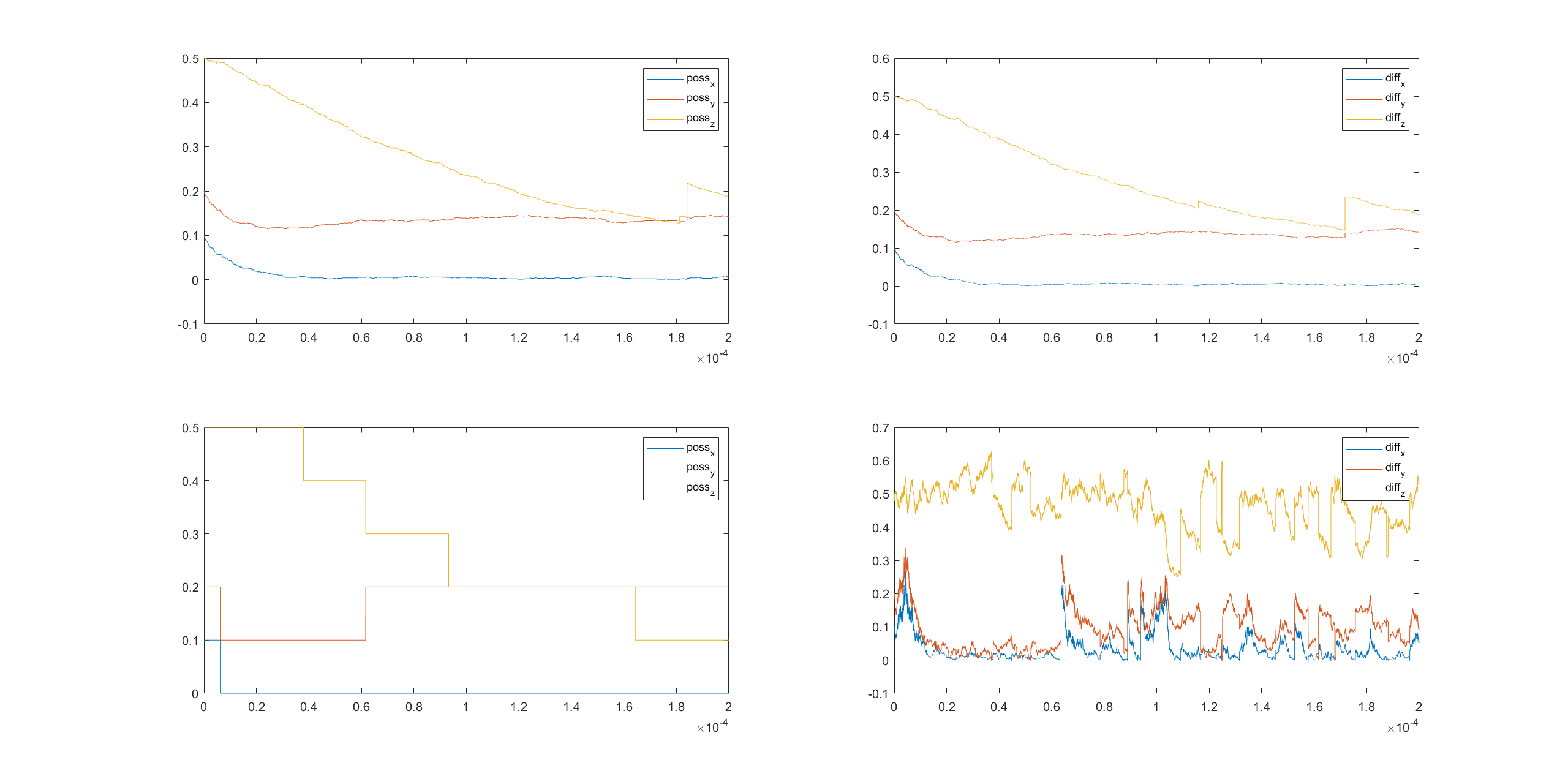}
     \caption{($V=1000$ vs. $V=10$) $\textbf{Upper panel}$: ($\textbf{Left}$) Trajectories of Poisson process for $V=1000$. ($\textbf{Right}$)Trajectories of  diffusion process for $V=1000$ .\ $\textbf{Lower panel}$: ($\textbf{Left}$) Trajectories of Poisson process for $V=10$. ($\textbf{Right}$) Trajectories of diffusion process for $V=10$.}     
    \label{F2}
 \end{figure} 

As we consider the effect of the capacity volume, the finite time error and the contraction rate for different volumes are compared in Table \ref{t1}. The last column $d_w(\hat{\pi}_X, \hat{\pi}_Y)$ is computed via \eqref{inequality}. It is not hard to see that upper bound of $d_w(\hat{\pi}_X, \hat{\pi}_Y)$ is quite larger when $V=10$. This is consistent with Theorem \ref{distance}, the supreme distance between two processes will be smaller as $V$ is getting larger. 

\begin{center}
\begin{table}
\begin{tabular}{| c |c| c| c|c| }
\hline
 volume $V$ & cut-off time $T$&finite time error & contraction rate $\gamma$ & $d_w(\hat{\pi}_X, \hat{\pi}_Y)$ \\
 \hline
 1000 & $2\times 10^{-4}$&0.0057 & 3.3616*$10^3$ & 0.0116 \\  
 400 & $4\times 10^{-5}$& 0.0088 & 2.0599*$10^4$ & 0.0157\\
 100 & $1\times 10^{-5}$ & 0.0099 & 6.0150*$10^4$  & 0.0195\\
 10 & $2\times 10^{-6}$ &0.0563 & 2.0927*$10^5$ & 0.1646\\
 \hline
\end{tabular}
\caption{Oregonator model: Numerical results for different volumes}
\label{t1}
\end{table}
\end{center}

\subsection{4D Lotka-Volterra Competitive Dynamics} Originally derived by Volterra in 1926 to describe the interaction between a predator species and a prey species \cite{lotka1926elements} and independently by Lotka to describe a chemical reaction \cite{volterra1926variazioni}, the general Lotka-Volterra model is widely used in ecology, biology, chemistry, physics, etc \cite{murray2001mathematical}. In this example we consider here a chaotic system in which 4 species with whole population $V$ compete for a finite set of resources. The ODE version of the system reads
\begin{equation*}
    \frac{d S_i}{d t} = r_i S_i (1 - \sum^4_{j=1}a_{i j} S_j), \ i = 1, 2, 3, 4.
\end{equation*}
Here $r_i$ represents the growth rate of species $i$ and $a_{i j}$ represents the extent to which species $j$ competes for resources used by species $i$. The parameter values are
\begin{equation*}
    r = \left(r_i\right)_{i=1}^4 = \begin{pmatrix}
    1\\
    0.72\\
    1.53\\
    1.27
    \end{pmatrix}, \ A = \left(a_{i j}\right)_{i,j=1}^4=\begin{pmatrix}
    1 & 1.09 & 1.52 & 0\\
    0 & 1 & 0.44 & 1.36\\
    2.33 & 0 & 1 & 0.47\\
    1.21 & 0.51 & 0.35 & 1
    \end{pmatrix}
\end{equation*}

For $i=1,\cdots,4$, all reactions in this system are shown as follows. 
\begin{equation*}
    S_i\overset{r_i}{\Rightarrow} 2S_i,\ S_1+S_i\overset{a_{i 1}r_i}{\Rightarrow}S_1,\ 
    S_2 + S_i \overset{a_{i 2}r_i}{\Rightarrow}  S_2, \ 
    S_3+S_i\overset{a_{i 3}r_i}{\Rightarrow} S_3, \ 
    S_4+S_i\overset{a_{i 4}r_i}{\Rightarrow}S_4
\end{equation*}

The corresponding rate functions are 
\begin{equation*}
\begin{aligned}
 q^n_{i,1} &= \sum^{n-1}_{m=0}V h r_i S_{i,m} \\
 q^n_{i,2} &= \sum^{n-1}_{m=0}V h r_i a_{i 1}S_{1,m}S_{i,m} \\
 q^n_{i,3} &= \sum^{n-1}_{m=0}V h r_i a_{i 2}S_{2,m}S_{i,m} \\
 q^n_{i,4} &= \sum^{n-1}_{m=0}V h r_i a_{i 3}S_{3,m}S_{i,m} \\
 q^n_{i,5} &= \sum^{n-1}_{m=0}V h r_i a_{i 4}S_{4,m}S_{i,m} 
\end{aligned}
\end{equation*}

As three zeros appear in coefficient matrix A, this system actually include 17 reactions. Therefore, the Poisson process $\hat{X}_n = (S_{1,n}, S_{2,n}, S_{3,n}, S_{4,n})$ can be written as
\begin{equation*}
\hat{X}_{n+1} =\begin{pmatrix}
    S_{1, n+1}\\
    S_{2, n+1}\\
    S_{3, n+1}\\
    S_{4, n+1}
    \end{pmatrix} =\begin{pmatrix}
    S_{1,n}\\
    S_{2,n}\\
    S_{3,n}\\
    S_{4,n}
    \end{pmatrix}\\
    +\frac{1}{V} \begin{pmatrix}
     \bm{f_1}(P_1, \cdots, P_{17}, q^n_{i, 1}, \cdots, q^{n}_{i, 5})\\
     \bm{f_2}(P_1, \cdots, P_{17}, q^n_{i, 1}, \cdots, q^{n}_{i, 5} )\\
      \bm{f_3}(P_1, \cdots, P_{17}, q^n_{i, 1}, \cdots, q^{n}_{i, 5})\\
      \bm{f_4}(P_1, \cdots, P_{17}, q^n_{i, 1}, \cdots, q^{n}_{i, 5})
    \end{pmatrix}, 
\end{equation*}
where $i=1, \cdots, 4$, $P_j,\  j=\{1, \cdots, 17\}$ are independent unit rate Poisson processes, $\bm{f_1}$, $\bm{f_2}$, $\bm{f_3}$ and $\bm{f_4}$ comes from discrete approximation in equation $\eqref{possN}$. To improve the
readability of the present paper, we move detailed expressions of $\bm{f_1}$ to $\bm{f_4}$ to the appendix.

The diffusion  approximation $\hat{Y}$ can be written as
\begin{equation}
\label{4DLV_Y1}
    \hat{Y}_{n+1} =\begin{pmatrix}
    S_{1, n+1}\\
    S_{2, n+1}\\
    S_{3, n+1}\\
    S_{4, n+1}
    \end{pmatrix} =\begin{pmatrix}
    S_{1,n}\\
    S_{2,n}\\
    S_{3,n}\\
    S_{4,n}
    \end{pmatrix}\\
    +\frac{1}{V} \begin{pmatrix}
     \bm{g_1}(q^n_{i, 1}, \cdots, q^{n}_{i, 5})\\
     \bm{g_2}(q^n_{i, 1}, \cdots, q^{n}_{i, 5} )\\
      \bm{g_3}(q^n_{i, 1}, \cdots, q^{n}_{i, 5})\\
      \bm{g_4}(q^n_{i, 1}, \cdots, q^{n}_{i, 5})
    \end{pmatrix}
    + \frac{1}{V} \begin{pmatrix}
     \bm{\sigma_1}(B_1, \cdots, B_{17}, q^n_{i, 1}, \cdots, q^{n}_{i, 5})\\
     \bm{\sigma_2}(B_1, \cdots, B_{17}, q^n_{i, 1}, \cdots, q^{n}_{i, 5} )\\
      \bm{\sigma_3}(B_1, \cdots, B_{17}, q^n_{i, 1}, \cdots, q^{n}_{i, 5})\\
      \bm{\sigma_4}(B_1, \cdots, B_{17}, q^n_{i, 1}, \cdots, q^{n}_{i, 5})
    \end{pmatrix},
\end{equation}
where $i=1, \cdots, 4$, $B_j, \ j=\{1, \cdots, 17\}$ are independent standard Wiener process, functions $\bm{g_1}$, $\bm{g_2}$, $\bm{g_3}$, $\bm{g_4}$, $\bm{\sigma_1}$, $\bm{\sigma_2}$, $\bm{\sigma_3}$ and $\bm{\sigma_4}$ follows the expression in equation $\eqref{diffN}$. We refer the
appendix for the detailed form of these functions.

By the stationary increments property and independence of Wiener processes $B_i, i= \{1, \cdots, 6\}$, equation \eqref{4DLV_Y1} can be simplified to: 
\begin{equation}
\label{4DLV_diff_matrix}
\begin{aligned}
    \hat{Y}_{n+1} &=\begin{pmatrix}
    S_{1, n+1}\\
    S_{2, n+1}\\
    S_{3, n+1}\\
    S_{4, n+1}
    \end{pmatrix} = \begin{pmatrix}
    S_{1,n}\\
    S_{2,n}\\
    S_{3,n}\\
    S_{4,n}
    \end{pmatrix} + \frac{1}{V}\begin{pmatrix}
    {\bm g}_1 (q_1, \cdots, q_4)\\
    {\bm g}_2 (q_1, \cdots, q_4)
    \end{pmatrix} 
    + \frac{1}{V} M \begin{pmatrix}
     W_1 \\ W_2 \\ \vdots \\W_{16} \\ W_{17}
    \end{pmatrix}
\end{aligned}
\end{equation}
where $W_i, i = 1, \cdots, 17$ are independent standard normal random variables, and $M$ is a matrix that depends only on $S_n$ and $I_n$. We refer readers to the appendix for the full expression of $M$.

Let $\partial\mathcal{X}$ be union of 4 axes. Processes $\hat{X}$ and $\hat{Y}$ admit QSDs $\hat{\pi}_X$ and $\hat{\pi}_Y$, respectively.
Long trajectories $P(i\Delta)$ and $B(i\Delta)$ for $i=\{1,\cdots, 2^{22}\}$ and $\Delta=0.01$ are constructed when we consider the trajectory-by-trajectory behaviour of two processes. The time step size is $h=0.001$ and the fixed time is set as $T=1$. 

Figure 5 Top Left shows the solution of the ordinary differential equation projected onto $x_1x_2x_3$ space. The trajectories of Poisson process and the diffusion process are shown on the Top Right and Bottom Left. It looks that the trajectories are close and this is reasonable because with high probability, the trajectories of both Poisson process and the diffusion process moves far away from the absorbing set $\partial\mathcal{X}$. We compute the distribution of the coupling time. The coupling time distribution and its exponential tail are shown in Figure 5 Top Left, that gives the contraction rate $\gamma=0.0849$. Then we apply Algorithm 1 to compute the finite time error. The finite time error is 0.0030 for $V = 1000$. As a result, the upper bound given in \eqref{inequality} is 0.0375 for $V = 1000$. 

To compare the different situations for volume $V=1000$ and $V=10$,  we plot trajectories of each species for $V=1000$ in Figure 6, and the case for $V=10$ is shown in Figure 7. It is not hard to see the trajectory-by-trajectory behavior between Poisson process and diffusion process is quite remarkable when $V=1000$. However, more regeneration happens in Poisson process when $V=10$. So it's not surprised us that the finite time error for $V=10$ is 0.1286, that around 40 times larger than the case for $V=1000$. Trajectories of the Poisson process have high probability moving along the boundary in this case. Same as above, we compute the contraction rate $\gamma$ of the coupling time distribution to be $1.7905$. As a result, the upper bound given in \eqref{inequality} is 0.1543 for $V=10$. 

\begin{figure}[h]
 \centering
     \includegraphics[width=\textwidth]{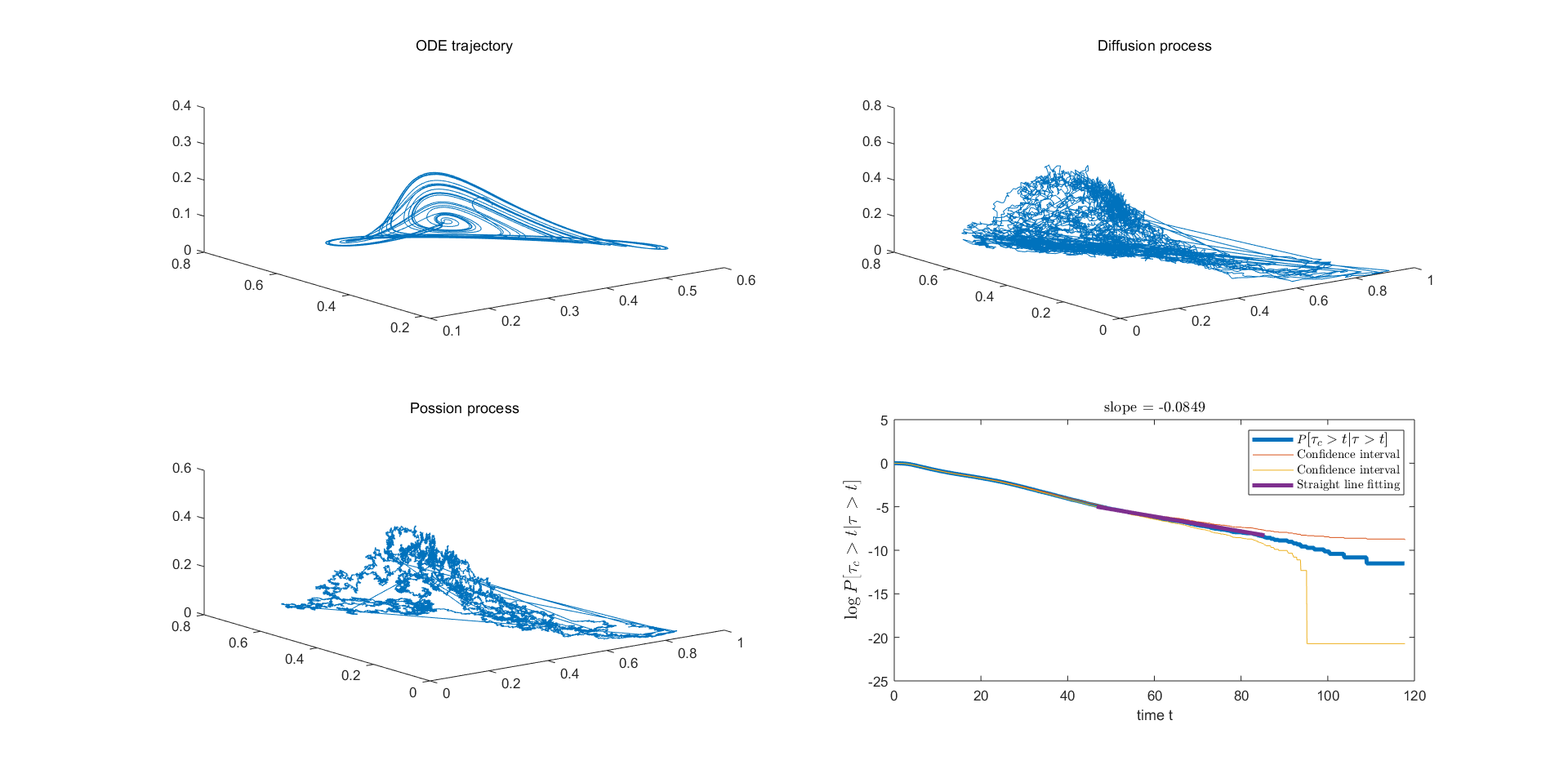}
     \caption{(Case $V=1000$) $\textbf{Upper panel}$: ($\textbf{Left}$) ODE trajectories. ($\textbf{Right}$) Poisson process.\ $\textbf{Lower panel}$: ($\textbf{Left}$) Diffusion process. ($\textbf{Right}$) $\mathbb{P}(\tau_c>t|\tau<t)$ vs.$t$.}     
    \label{F1}
 \end{figure}
 
\begin{figure}[h]
 \centering
     \includegraphics[width=\textwidth]{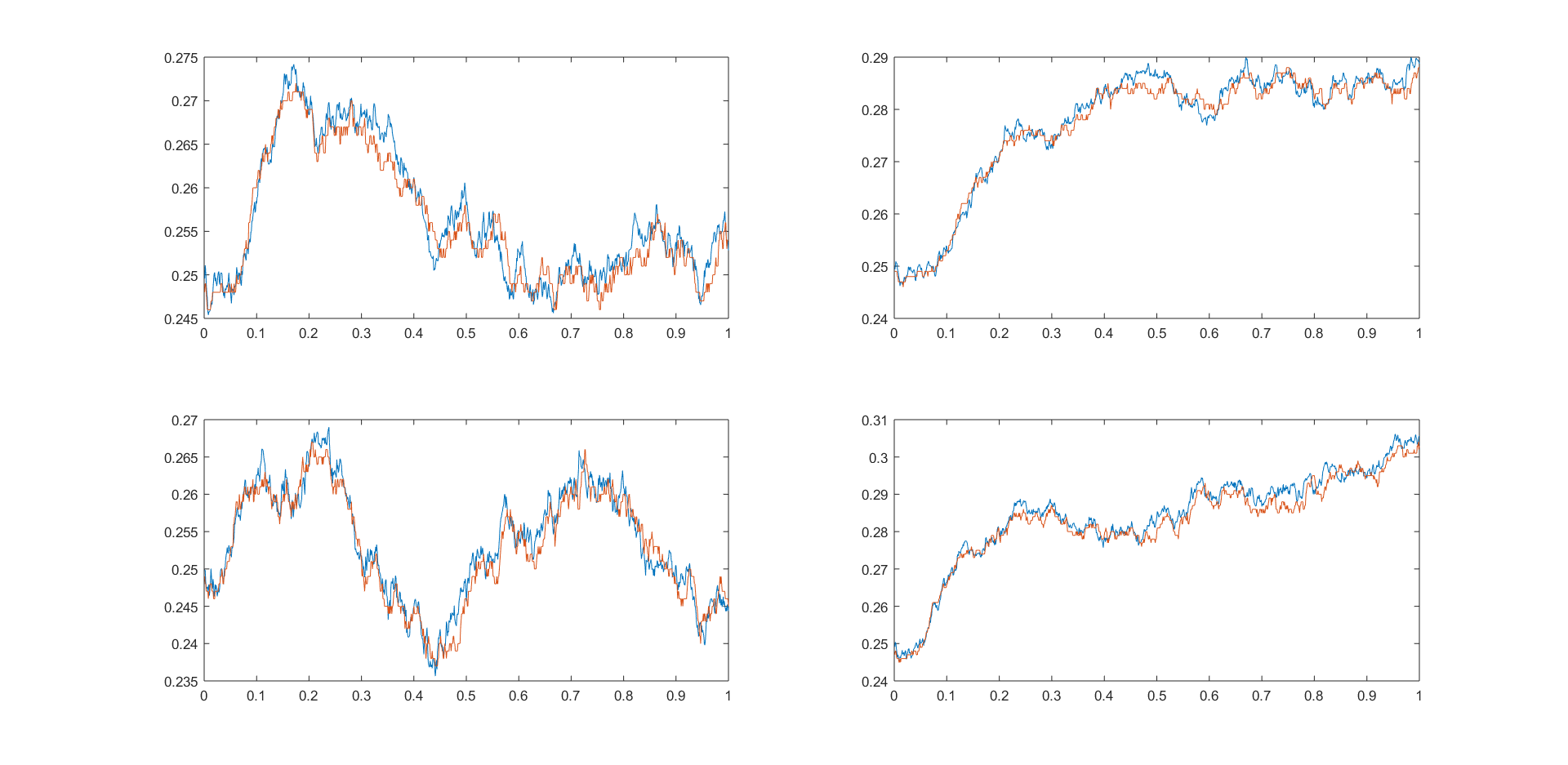}
     \caption{(Case $V=1000$) Poisson trajectories and diffusion trajectories for 4 species .}     
    \label{F2}
 \end{figure} 
 
\begin{figure}[h]
 \centering
     \includegraphics[width=\textwidth]{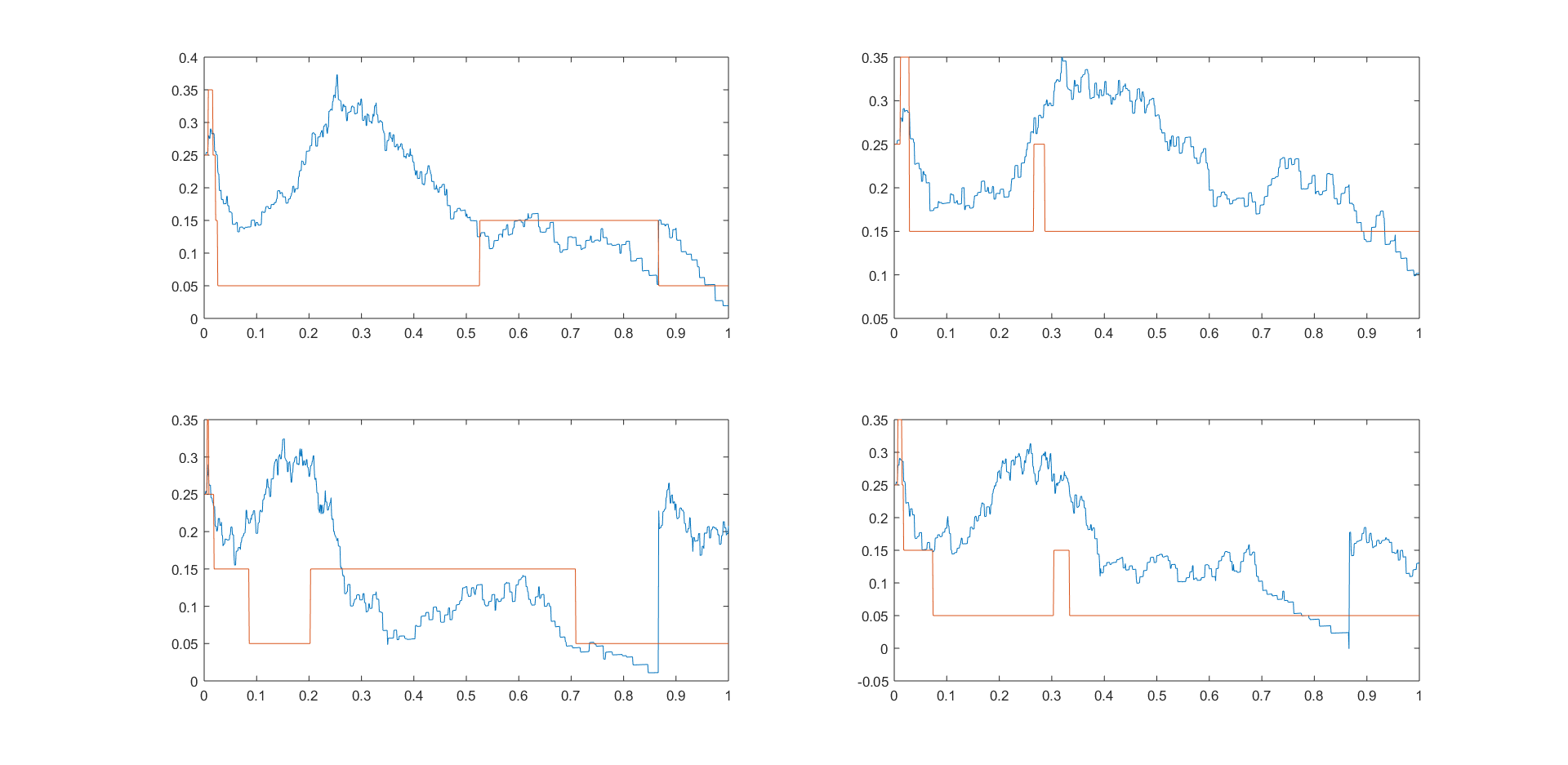}
     \caption{(Case $V=10$) Poisson trajectories and diffusion trajectories for 4 species.}     
    \label{F2}
 \end{figure} 

As we consider the effect of the capacity volume, the finite time error and the contraction rate for different volumes are compared in Table \ref{t1}. The last column $d_w(\hat{\pi}_X, \hat{\pi}_Y)$ is computed via \eqref{inequality}. It is not hard to see that upper bound of $d_w(\hat{\pi}_X, \hat{\pi}_Y)$ is quite larger when $V=10$. This is consistent with Theorem \ref{distance}, the supreme distance between two processes will be smaller as $V$ is getting larger. 

\begin{center}
\begin{table}
\begin{tabular}{| c |c| c| c| }
\hline
 volume $V$ & finite time error & contraction rate $\gamma$ & $d_w(\hat{\pi}_X, \hat{\pi}_Y)$ \\
 \hline
 1000 & 0.0030 & 0.0849 & 0.0375 \\  
 400 & 0.0110 & 0.1831 & 0.0659\\
 100 & 0.0502 & 0.3110  & 0.1878\\
 10 & 0.1286 & 1.7905 & 0.1543\\
 \hline
\end{tabular}
\caption{4D Lotka-Volterra model: Numerical results for different volumes}
\label{t1}
\end{table}
\end{center}

\section{Conclusion}
In this paper we develop a coupling-based approach to quantitatively estimate the distance between the QSD of a stochastic mass-action process and that of its diffusion approximation. The dependence of QSDs in terms of the volume of the mass-action system is studied. To address the challenge of QSDs, we use the idea of regeneration from QSDs after exiting to construct a process with stationary distribution. This is the the main change from our previous work  \cite{li2020numerical,dobson2021using}. Both the coupling algorithm and the path-wise matching of a stochastic mass-action system and its diffusion approximation need to be adapted to the regeneration from QSDs.  We compare the finite time error and the rate of contraction for different population size $V$. All numerical results shows that the distance between two QSDs is smaller for larger population. In general, the effect of demographic noise must be seriously addressed when the population is small.

The study of path-wise approximation of stochastic mass-action systems by diffusion processes and the coupling of diffusion processes motivates a very interesting question. All our existing work relies on the reflection coupling of diffusion processes, which is known to be highly effective. Then how can one effectively couple two continuous-time Markov processes on a lattice? A successful coupling of two trajectories of a mass-action system will extend our framework of sensitivity analysis to many more applications. We believe it is very difficult to couple the exact stochastic mass-action system because random events occur at continuous time. However, there may be  some way of building a "discrete reflection" and coupling two tau-leaping trajectories, i.e., two trajectories of equation \eqref{possN} effectively. This will be addressed in our future work. 

\appendix

\section{Expressions of mass-action systems and their diffusion approximations}
To improve the readability, we put the explicit formulas of the Poisson approximation and the diffusion approximation for each model in this section. 
\subsection{SIR model}
There are four reactions are involved in the SIR system, so we have 4 pairs of Poisson process $P_i$ and Wiener process $B_i$ appear in the evolution of each class. The rule of update of the numerical approximation $\hat{X}_n$ follows
\begin{tiny}
\begin{equation*}
\begin{aligned}
    \hat{X}_{n+1} &=\begin{pmatrix}
    S_{n+1}\\
    I_{n+1}
    \end{pmatrix} \\
    & = \begin{pmatrix}
    S_n\\
    I_n
    \end{pmatrix} + \frac{1}{V}\begin{pmatrix}
    \left[P_1(q_{1,n+1}) - P_1(q_{1,n})\right] - \left[P_2(q_{2,n+1}) - P_2(q_{2,n}\right] - \left[P_3(q_{3,n+1} - P_3(q_{3,n})\right]\\
    \left[P_2(q_{2,n+1}) - P_2(q_{2,n}) \right] - \left[P_4(q_{4,n+1}) - P_4(q_{4,n})\right]
    \end{pmatrix},\\
    &:=\begin{pmatrix}
    S_n\\
    I_n
    \end{pmatrix} + \frac{1}{V}\begin{pmatrix}
   {\bm f}_1(P_1, \cdots, P_4, q_{1,n}, \cdots , q_{4,n})\\
    {\bm f}_2(P_1, \cdots, P_4, q_{1,n}, \cdots, q_{4,n}) 
    \end{pmatrix}
\end{aligned}
\end{equation*}
\end{tiny}
where $P_i, i=\{1,2,3,4\}$ are independent unit rate Poisson processes. 

The rule of update of the numerical approximation $\hat{Y}_n$ follows
\begin{tiny}
\begin{equation*}
\begin{aligned}
    \hat{Y}_{n+1} &=\begin{pmatrix}
    S_{n+1}\\
    I_{n+1}
    \end{pmatrix} = \begin{pmatrix}
    S_n\\
    I_n
    \end{pmatrix} + \frac{1}{V}\begin{pmatrix}
    [q_{1,n+1}-q_{1,n}] - [q_{2,n+1}-q_{2,n}] -[q_{3,n+1}-q_{3,n}]\\
    [q_{2,n+1}-q_{2,n}] - [q_{4,n+1}-q_{4,n}]
    \end{pmatrix} \\ 
    & + \frac{1}{V}\begin{pmatrix}
    \left[B_1(q_{1,n+1}) - B_1(q_{1,n})\right] - \left[B_2(q_{2,n+1}) - B_2(q_{2,n})\right] - \left[B_3(q_{3,n+1} - B_3(q_{3,n}\right]\\
    \left[B_2(q_{2,n+1}) - B_2(q_{2,n}) \right] - \left[B_4(q_{4,n+1}) - B_4(q_{4,n})\right]
    \end{pmatrix},\\
    & := \begin{pmatrix}
    S_n\\
    I_n
    \end{pmatrix} + \frac{1}{V}\begin{pmatrix}
    {\bm g}_1 (q_{1,n}, \cdots , q_{4,n})\\
    {\bm g}_2 (q_{1,n}, \cdots , q_{4,n})
    \end{pmatrix} 
    + \frac{1}{V}\begin{pmatrix}
    {\bm \sigma}_1(  B_1, \cdots, B_4, q_{1,n}, \cdots , q_{4,n})\\
    {\bm \sigma}_2(B_1, \cdots, B_4, q_{1,n}, \cdots , q_{4,n})
    \end{pmatrix},
\end{aligned}
\end{equation*}
\end{tiny}
where $B_i,\ i=\{1, 2, 3, 4\}$ are independent standard Wiener processs.

As two classes $S_n$ and $I_n$ and four reactions are considered in this SIR model, the corresponding diffusion matrix $M$ should be a $2\times4$ matrix. Specifically, the diffusion matrix $M$ reads as
\begin{tiny}
\begin{equation*}
\begin{aligned}
    \hat{Y}_{n+1} &=\begin{pmatrix}
    S_{n+1}\\
    I_{n+1}
    \end{pmatrix} = \begin{pmatrix}
    S_n\\
    I_n
    \end{pmatrix} + \frac{1}{V}\begin{pmatrix}
    [q_{1,n+1}-q_{1,n}] - [q_{2,n+1}-q_{2,n}] -[q_{3,n+1}-q_{3,n}]\\
    [q_{2,n+1}-q_{2,n}] - [q_{4,n+1}-q_{4,n}]
    \end{pmatrix} \\ 
    & + \frac{1}{V}\begin{pmatrix}
    \sqrt{q_{1,n+1}-q_{1,n}} & -\sqrt{q_{2,n+1}-q_{2,n}} &
    -\sqrt{q_{3,n+1}-q_{3,n}} & 0 \\
    0 & \sqrt{q_{2,n+1}-q_{2,n}} & 0 &
    -\sqrt{q_{4,n+1}-q_{4,n}}
    \end{pmatrix} 
    \begin{pmatrix}
     W_1 \\ W_2 \\ W_3 \\ W_4
    \end{pmatrix}\\
    & = \begin{pmatrix}
    S_n\\
    I_n
    \end{pmatrix} + \frac{1}{V}\begin{pmatrix}
    V h\alpha- V h \beta S_n I_n  -V h \mu S_n\\
    V h \beta S_n I_n - V h (\mu+\rho+\gamma) I_n
    \end{pmatrix}  \\ 
    & + \frac{1}{V}\begin{pmatrix}
     \sqrt{V h\alpha} & -\sqrt{V h \beta S_n I_n } & -\sqrt{V h \mu S_n} & 0\\
     0 & \sqrt{V h \beta S_n I_n} & 0 & -\sqrt{V h (\mu+\rho+\gamma) I_n}
    \end{pmatrix}\begin{pmatrix}
     W_1 \\ W_2 \\ W_3 \\ W_4
    \end{pmatrix}\\
    & := \begin{pmatrix}
    S_n\\
    I_n
    \end{pmatrix} + \frac{1}{V}\begin{pmatrix}
    {\bm g}_1 (q_{1,n}, \cdots , q_{4,n})\\
    {\bm g}_2 (q_{1,n}, \cdots , q_{4,n})
    \end{pmatrix} 
    + \frac{1}{V} M \begin{pmatrix}
     W_1 \\ W_2 \\ W_3 \\ W_4
    \end{pmatrix}
\end{aligned}
\end{equation*}
\end{tiny}
where $W_i,\ i=\{1, 2, 3, 4\}$ are independent standard normal distributed random variables. 

\subsection{Oregnator model}
For the Oregnator model, there are six reactions involved. So we have 6 pairs of Poisson process $P_i$ and $B_i$ in the approximations. The rule of update of the numerical approximation $\hat{X}_n$ follows
\begin{tiny}
\begin{equation*}
\begin{aligned}
    &\hat{X}_{n+1} =\begin{pmatrix}
    S_{1,n+1}\\
    S_{2, n+1}\\
    S_{3, n+1}
    \end{pmatrix} =\begin{pmatrix}
    S_{1, n}\\
    S_{2, n}\\
    S_{3, n}
    \end{pmatrix}\\
    & + \frac{1}{V}\begin{pmatrix}
    [P_1(q_{1,n+1}) - P_1(q_{1,n})]- [P_2(q_{2,n+1}) - P_2(q_{2,n})]+[P_3(q_{3,n+1}) - P_3(q_{3,n})]-2[P_4(q_{4,n+1}) - P_4(q_{4,n})]\\
    - [P_1(q_{1, n+1}) - P_1(q_{1, n})] - [P_2(q_{2,n+1}) - P_2(q_{2,n})]+[P_5(q_{5, n+1}) - P_5(q_{5, n})]\\
    2[P_3(q_{3,n+1}) - P_3(q_{3,n}) ]-[P_5(q_{5,n+1}) - P_5(q_{5,n}) ]-[P_6(q_{6,n+1}) - P_6(q_{6,n})]
    \end{pmatrix}\\
    & := \begin{pmatrix}
    S_{1, n}\\
    S_{2, n}\\
    S_{3, n}
    \end{pmatrix}+ \frac{1}{V}\begin{pmatrix}
    \bm{f_1}(P_1,\cdots, P_6, q_{1, n},\cdots, q_{6,n})\\
    \bm{f_2}(P_1,\cdots, P_6, q_{1, n},\cdots, q_{6,n})\\
    \bm{f_3}(P_1,\cdots, P_6, q_{1, n},\cdots, q_{6,n})
    \end{pmatrix}
\end{aligned}
\end{equation*}
\end{tiny}
where $P_i, i=\{1,\cdots, 6\}$ are independent unite rate Poisson processes.

The diffusion approximation $\hat{Y}$ can be written as
\begin{tiny}
\begin{equation*}
\begin{aligned}
    &\hat{Y}_{n+1} = \begin{pmatrix}
    S_{1,n+1}\\
    S_{2,n+1}\\
    S_{3,n+1}
    \end{pmatrix} \\
    &=\begin{pmatrix}
    S_{1,n}\\
    S_{2,n}\\
    S_{3,n}
    \end{pmatrix} +
   \frac{1}{V}\begin{pmatrix}
   [q_{1,n+1} - q_{1,n}] - [q_{2,n+1} - q_{2,n}] + [q_{3,n+1}-q_{3,n}] - 2[q_{4,n+1}-q_{4,n}]\\
   -[q_{1,n+1} -q_{1,n}] -[q_{2,n+1} - q_{2,n}] + [q_{5,n+1}-q_{5,n}]\\
   2[q_{3,n+1}-q_{3,n}] - [q_{5,n+1}-q_{5,n}] - [q_{6,n+1}-q_{6,n}]
   \end{pmatrix}\\
    &+ \frac{1}{V}\begin{pmatrix}
    [B_1(q_{1,n+1}) - B_1(q_{1,n})]- [B_2(q_{2,n+1}) - B_2(q_{2,n})]+[B_3(q_{3,n+1}) - B_3(q_{3,n})]-2[B_4(q_{4,n+1}) - B_4(q_{4,n})]\\
    - [B_1(q_{1, n+1}) - B_1(q_{1, n})] - [B_2(q_{2,n+1}) - B_2(q_{2,n})]+[B_5(q_{5, n+1}) - B_5(q_{5, n})]\\
    2[B_3(q_{3,n+1}) - B_3(q_{3,n}) ]-[B_5(q_{5,n+1}) - B_5(q_{5,n}) ]-[B_6(q_{6,n+1}) - B_6(q_{6,n})]
    \end{pmatrix}\\
    & :=\begin{pmatrix}
    S_{1,n}\\
    S_{2,n}\\
    S_{3,n}
    \end{pmatrix} + \frac{1}{V}\begin{pmatrix}
     \bm{g_1}(q_{1,n}, \cdots, q_{6,n})\\
     \bm{g_2}(q_{1, n}, \cdots, q_{6, n})\\
     \bm{g_3}(q_{1, n}, \cdots, q_{6, n})
    \end{pmatrix} +
   \frac{1}{V}\begin{pmatrix} \bm{\sigma_1}(B_1,\cdots, B_6, q_{1, n}, \cdots, q_{6, n})\\
    \bm{\sigma_2}(B_1,\cdots, B_6, q_{1, n}, \cdots, q_{6, n})\\
     \bm{\sigma_3}(B_1,\cdots, B_6, q_{1, n}, \cdots, q_{6, n})
   \end{pmatrix}, 
\end{aligned}
\end{equation*}
\end{tiny}
where $B_i, \{i=1, \cdots, 6\}$ are independent Wiener processes.

As we focus on three classes $S_{1,n}, S_{2,n}, S_{3,n}$ and six reactions, we can confirm that the diffusion matrix $M$ is a $3\times6$ matrix. Specifically, the diffusion matrix $M$ is defined as follows.
\begin{tiny}   
\begin{equation*}  
\begin{aligned}
   &\hat{Y}_{n+1} = \begin{pmatrix}
    S_{1,n+1}\\
    S_{2,n+1}\\
    S_{3,n+1}
    \end{pmatrix} \\
    &=\begin{pmatrix}
    S_{1,n}\\
    S_{2,n}\\
    S_{3,n}
    \end{pmatrix} +
   \frac{1}{V}\begin{pmatrix}
   [q_{1,n+1} - q_{1,n}] - [q_{2,n+1} - q_{2,n}] + [q_{3,n+1}-q_{3,n}] - 2[q_{4,n+1}-q_{4,n}]\\
   -[q_{1,n+1} -q_{1,n}] -[q_{2,n+1} - q_{2,n}] + [q_{5,n+1}-q_{5,n}]\\
   2[q_{3,n+1}-q_{3,n}] - [q_{5,n+1}-q_{5,n}] - [q_{6,n+1}-q_{6,n}]
   \end{pmatrix}\\
   & + \frac{1}{V}\begin{pmatrix}
    \sqrt{q_{1,n+1} - q_{1,n}} & -\sqrt{q_{2,n+1} - q_{2,n}} & \sqrt{q_{3,n+1}-q_{3,n}} & -\sqrt{2[q_{4,n+1}-q_{4,n}]} & 0 &  0\\
    -\sqrt{q_{1,n+1} -q_{1,n}} & -\sqrt{q_{2,n+1} - q_{2,n}} & 0 & 0 & \sqrt{q_{5,n+1}-q_{5,n}} & 0\\
   0 & 0 & \sqrt{2[q_{3,n+1}-q_{3,n}]} & 0 & -\sqrt{q_{5,n+1}-q_{5,n}} & -\sqrt{q_{6,n+1}-q_{6,n}}
   \end{pmatrix}\begin{pmatrix}
        W_1\\
        W_2\\
        W_3\\
        W_4\\
        W_5\\
        W_6
    \end{pmatrix}\\
   &=\begin{pmatrix}
    S_{1,n}\\
    S_{2,n}\\
    S_{3,n}
    \end{pmatrix} + \frac{1}{V}\begin{pmatrix}
   V h C_1 S_{2,n} - V h C_2 S_{1,n}S_{2,n} + V h C_3 S_{1,n} - 2V h C_4S_{1,n}^2\\
    V h C_1 S_{2,n} - V h C_2 S_{1,n}S_{2,n}  +V h C_5 \delta S_{3,n}\\
   2V h C_3 S_{1,n} - V h C_5 \delta S_{3,n} - V h C_5 (1-\delta) S_{3,n}
   \end{pmatrix}\\
   &+ \frac{1}{V}\begin{pmatrix}
        \sqrt{V h C_1 S_{2,n}} & -\sqrt{V h C_2 S_{1,n}S_{2,n}} & \sqrt{V h C_3 S_{1,n}} & 2\sqrt{V h C_4S_{1,n}^2} & 0 & 0\\
        -\sqrt{V h C_1 S_{2,n}} & -\sqrt{V h C_2 S_{1,n}S_{2,n}} & 0 & 0 & \sqrt{V h C_5 \delta S_{3,n}} & 0\\
        0 & 0 & 2\sqrt{V h C_3 S_{1,n}} & 0 & -\sqrt{V h C_5 \delta S_{3,n}} & -\sqrt{V h C_5 (1-\delta) S_{3,n}}
    \end{pmatrix}\begin{pmatrix}
        W_1\\
        W_2\\
        W_3\\
        W_4\\
        W_5\\
        W_6
    \end{pmatrix}\\
    & := \begin{pmatrix}
    S_{1,n}\\
    S_{2,n}\\
    S_{3,n}
    \end{pmatrix} + \frac{1}{V}\begin{pmatrix}
     \bm{g_1}(q_{1,n}, \cdots, q_{6,n})\\
     \bm{g_2}(q_{1, n}, \cdots, q_{6, n})\\
     \bm{g_3}(q_{1, n}, \cdots, q_{6, n})
    \end{pmatrix} +
   \frac{1}{V} M \begin{pmatrix}
        W_1\\
        W_2\\
        W_3\\
        W_4\\
        W_5\\
        W_6
    \end{pmatrix},
\end{aligned}
\end{equation*}
\end{tiny}
where $W_i, \{i=1, \cdots, 6\}$ are independent standard normal distributed random variables.
\subsection{4D Lotka-Volterra model}
For the 4D Lotka-Volterra system, there are 17 reactions involved, so we have 17 pairs of Poisson process $P_i$ and Wiener process $B_i$. The rule of update of the numerical approximation $\hat{X}_n$ follows
\begin{tiny}
\begin{equation*}
\begin{aligned}
    &\hat{X}_{n+1} =\begin{pmatrix}
    S_{1, n+1}\\
    S_{2, n+1}\\
    S_{3, n+1}\\
    S_{4, n+1}
    \end{pmatrix} =\begin{pmatrix}
    S_{1,n}\\
    S_{2,n}\\
    S_{3,n}\\
    S_{4,n}
    \end{pmatrix}\\
    &+\frac{1}{V}\begin{psmallmatrix}
     [P_1(q_{1,n+1})-P_1(q_{1,n})]-[P_2(q_{2,n+1})-P_2(q_{2,n})] -[P_3(q_{1,n+1}) - P_3(q_{1,n})] -[P_4(q_{4,n+1})-P_4(q_{4,n})] \\
     [P_5(q_{5,n+1})-P_5(q_{5,n})]-[P_6(q_{6,n+1})-P_6(q_{6,n})] -[P_7(q_{7,n+1}) - P_7(q_{7,n})] -[P_8(q_{8,n+1})-P_8(q_{8,n})] \\
     [P_9(q_{9,n+1})-P_9(q_{9,n})]-[P_{10}(q_{10,n+1})-P_{10}(q_{10,n})] -[P_{11}(q_{11,n+1}) - P_{11}(q_{11,n})] -[P_{12}(q_{12,n+1})-P_{12}(q_{12,n})]\\
     [P_{13}(q_{13,n+1})-P_{13}(q_{13,n})]-[P_{14}(q_{14,n+1})-P_{14}(q_{14,n})] -[P_{15}(q_{15,n+1}) - P_{15}(q_{15,n})] -[P_{16}(q_{16,n+1})-P_{16}(q_{16,n})] -[P_{17}(q_{17,n+1})-P_{17}(q_{17,n})]
    \end{psmallmatrix}\\
    & := \begin{pmatrix}
    S_{1,n}\\
    S_{2,n}\\
    S_{3,n}\\
    S_{4,n}
    \end{pmatrix}
    +\frac{1}{V} \begin{pmatrix}
     \bm{f_1}(P_1, \cdots, P_{17}, q^n_{i, 1}, \cdots, q^{n}_{i, 5})\\
     \bm{f_2}(P_1, \cdots, P_{17}, q^n_{i, 1}, \cdots, q^{n}_{i, 5} )\\
      \bm{f_3}(P_1, \cdots, P_{17}, q^n_{i, 1}, \cdots, q^{n}_{i, 5})\\
      \bm{f_4}(P_1, \cdots, P_{17}, q^n_{i, 1}, \cdots, q^{n}_{i, 5})
    \end{pmatrix}, 
\end{aligned}
\end{equation*}
\end{tiny}
where $i=1, \cdots, 4$ and $P_j, \{j=1, \cdots, 17\}$ are independent unit rate Poisson processes. 
The diffusion approximation $\hat{Y}$ can be written as
\begin{tiny}
\begin{equation*}
\begin{aligned}
    &\hat{Y}_{n+1} = \begin{pmatrix}
    S_{1, n+1}\\
    S_{2, n+1}\\
    S_{3, n+1}\\
    S_{4, n+1}
    \end{pmatrix} =\begin{pmatrix}
    S_{1,n}\\
    S_{2,n}\\
    S_{3,n}\\
    S_{4,n}
    \end{pmatrix} + \frac{1}{V}\begin{psmallmatrix}
     [q_{1,n+1}-q_{1,n}] - [q_{2,n+1}-q_{2,n}] - [q_{3,n+1}-q_{3,n}] -[q_{4,n+1}-q_{4,n}]\\
     [q_{5,n+1}-q_{5,n}] -[q_{6,n+1}-q_{6,n}] - [q_{7,n+1}-q_{7,n}] -[q_{8,n+1}-q_{8,n}]\\
     [q_{9,n+1}-q_{9,n}] -[q_{10,n+1}-q_{10,n}] -[q_{11,n+1}-q_{11,n}] -[q_{12,n+1}-q_{12,n}]\\
     [q_{13,n+1}-q_{13,n}] - [q_{14,n+1} -q_{14,n}] - [q_{15,n+1}-q_{15,n}] -[q_{16,n+1}-q_{16,n}] - [q_{17,n+1}-q_{17,n}]
    \end{psmallmatrix}\\
    & + \frac{1}{V}\begin{psmallmatrix}
     [B_1(q_{1,n+1})-B_1(q_{1,n})]-[B_2(q_{2,n+1})-B_2(q_{2,n})] -[B_3(q_{3,n+1}) - B_3(q_{3,n})] -[B_4(q_{4,n+1})-B_4(q_{4,n})] \\
     [B_5(q_{5,n+1})-B_5(q_{5,n})]-[B_6(q_{6,n+1})-B_6(q_{6,n})] -[B_7(q_{7,n+1}) - B_7(q_{7,n})] -[B_8(q_{8,n+1})-B_8(q_{8,n})] \\
     [B_9(q_{9,n+1})-B_9(q_{9,n})]-[B_{10}(q_{10,n+1})-B_{10}(q_{10,n})] -[B_{11}(q_{11,n+1}) - B_{11}(q_{11,n})] -[B_{12}(q_{12,n+1})-B_{12}(q_{12,n})]\\
     [B_{13}(q_{13,n+1})-B_{13}(q_{13,n})]-[B_{14}(q_{14,n+1})-B_{14}(q_{14,n})] -[B_{15}(q_{15,n+1}) - B_{15}(q_{15,n})] -[B_{16}(q_{16,n+1})-B_{16}(q_{16,n})] -[B_{17}(q_{17,n+1})-B_{17}(q_{17,n})]
    \end{psmallmatrix}\\
    & :=\begin{pmatrix}
    S_{1,n}\\
    S_{2,n}\\
    S_{3,n}\\
    S_{4,n}
    \end{pmatrix}
    +\frac{1}{V} \begin{pmatrix}
     \bm{g_1}(q^n_{i, 1}, \cdots, q^{n}_{i, 5})\\
     \bm{g_2}(q^n_{i, 1}, \cdots, q^{n}_{i, 5} )\\
      \bm{g_3}(q^n_{i, 1}, \cdots, q^{n}_{i, 5})\\
      \bm{g_4}(q^n_{i, 1}, \cdots, q^{n}_{i, 5})
    \end{pmatrix}
    + \frac{1}{V} \begin{pmatrix}
     \bm{\sigma_1}(B_1, \cdots, B_{17}, q^n_{i, 1}, \cdots, q^{n}_{i, 5})\\
     \bm{\sigma_2}(B_1, \cdots, B_{17}, q^n_{i, 1}, \cdots, q^{n}_{i, 5} )\\
      \bm{\sigma_3}(B_1, \cdots, B_{17}, q^n_{i, 1}, \cdots, q^{n}_{i, 5})\\
      \bm{\sigma_4}(B_1, \cdots, B_{17}, q^n_{i, 1}, \cdots, q^{n}_{i, 5})
    \end{pmatrix},
\end{aligned}
\end{equation*}
\end{tiny}
where $i=1, \cdots, 4$, $B_j, \{j=1, \cdots, 17\}$ are independent Wiener processes.

\newpage
\bibliographystyle{plain}
\bibliography{ref.bib}

\end{document}